\documentclass{amsart}
\usepackage{mathrsfs}
\usepackage{amscd, amssymb, amsmath, amsthm, amsfonts}
\usepackage{latexsym, graphics, graphicx, psfrag}
\usepackage{color,xcolor,colortbl}
\usepackage[all]{xy}

\newtheorem{thm}{Theorem}[section]
\newtheorem{col}[thm]{Corollary}
\newtheorem{lem}[thm]{Lemma}
\newtheorem{prop}[thm]{Proposition}
\theoremstyle{definition}
\newtheorem{defn}[thm]{Definition}
\theoremstyle{remark}
\newtheorem{rem}[thm]{Remark}
\theoremstyle{example}

\numberwithin{equation}{section}

\usepackage[linkcolor=red, citecolor=blue]{hyperref}

\title{A Transcendental Invariant of Pseudo-Anosov Maps}

\author[Hongbin~Sun]{Hongbin Sun}
\address{%
    Department of Mathematics\\
    Princeton University\\
    Princeton, NJ 08544, USA}
\email{%
    hongbins@math.princeton.edu}

\date{}

\begin{document}

\begin{abstract}
   For each pseudo-Anosov map $\phi$ on surface $S$,
we will associate it with a $\mathbb{Q}$-submodule of $\mathbb{R}$, denoted by $A(S,\phi)$.
$A(S,\phi)$ is defined by an interaction between the Thurston norm and dilatation of pseudo-Anosov maps. We will develop a few nice properties of $A(S,\phi)$ and give a few examples to show that $A(S,\phi)$ is a nontrivial invariant. These nontrivial examples give an answer to a question asked by McMullen: the minimal point of the restriction of the dilatation function on fibered face need not be a rational point.
\end{abstract}

\maketitle

% ---------------------------------------------------------------

\section{Introduction}
\subsection{Background}
Given a pseudo-Anosov map $\phi$ on an oriented surface $S$ (possibly with boundary) with negative Euler
characteristic, a natural object associated with $(S,\phi)$ is the mapping torus $M(S,\phi)=S\times I/(x,0)\sim (\phi(x),1)$.
Thurston showed that $M(S,\phi)$ admits a
complete hyperbolic metric with finite volume (\cite{Ot}). Recently, Wise showed that every cusped hyperbolic $3$-manifold has
a finite cover which is a surface bundle over circle in \cite{Wi}, while Agol showed this result for closed hyperbolic $3$-manifold
(Virtually Fibered Conjecture) in \cite{Ag}. So hyperbolic surface bundles are virtually all finite volume hyperbolic
$3$-manifolds.

In this paper, we will always suppose $b_1(M(S,\phi))>1$. Then $N=M(S,\phi)$ has infinitely many different surface bundle over circle structures. These structures are organized by the Thurston norm on $H^1(N;\mathbb{R})$ (\cite{Th}). For any $\alpha \in H^1(N;\mathbb{Z})$
($\cong H_2(N,\partial N; \mathbb{Z})$ by duality), the Thurston norm of $\alpha$ is defined by:
$$\| \alpha \|=inf\{|\chi(T_0)|\ |\ (T,\partial T)\subset(N, \partial N)\ \text{is \ dual \ to}\ \alpha\},$$ here
$T_0\subset T$ excludes $S^2$ and $D^2$ components of $T$. Then the norm can be extended to $H^1(N;\mathbb{R})$
homogeneously and continuously, while the Thurston norm unit ball is a polyhedron with faces dual with elements in
$H_1(N;\mathbb{Z})/Tor$. For an open
face $F$ of the Thurston norm unit ball, let the open cone over $F$ denoted by $C$. Thurston showed that $C$
contains an integer cohomology class that corresponds to a surface bundle over circle structure if and only if all the cohomology classes
in $C$ correspond to such structures. In this case, the face $F$ is called a fibered face, and $C$ is called a fibered cone.

Another important object associated with pseudo-Anosov map $\phi$ is its dilatation $\lambda(\phi)\in \mathbb{R}_{>1}$.
For pseudo-Anosov map $\phi$,
there is a pair of transverse singular measured foliations (equivalently geodesic measured laminations) $(\mathcal{F}^+, \mu^+),
(\mathcal{F}^-,\mu^-)$, such that $\mathcal{F}^+$ and $\mathcal{F}^-$ are preserved by $\phi$ and
$\phi^*(\mu^+)=\lambda(\phi)\mu^+$, $\phi^*(\mu^-)=\frac{1}{\lambda(\phi)}\mu^-$.
So for any integer cohomology class $\alpha\in C\cap H^1(N; \mathbb{Z})$, it is associated with a number $\lambda(\alpha)$, which is the
dilatation of the corresponding monodromy map. For rational point $\alpha/n \in C\cap H^1(N; \mathbb{Q})$, we can define $\lambda(\alpha/n)=\lambda(\alpha)^n$,
so $\lambda(\cdot)$ is a function defined on $C\cap H^1(N; \mathbb{Q})$ now.

In \cite{Fr}, Fried showed that $\lambda(\cdot)$ can be extended to a continuous function on fibered cone $C$. We will use notation
$\lambda_C(\cdot)$ when we want to emphasize this function is defined on $C$. Fried also showed that $\lambda(\alpha)$ goes to infinity when $\alpha$
goes to the boundary of $C$, and $\frac{1}{\log{\lambda(\cdot)}}$ is a concave function on $C$ (see \cite{LO} for an alternative
proof). Moreover, Matsumoto showed that $\frac{1}{\log{\lambda(\cdot)}}$ is strictly concave along rays not going through the original
point (\cite{Ma}). It implies that the restriction of $\lambda(\cdot)$ on fibered face $F$ has a unique minimal point, which is denoted by $m_F$.

In \cite{McM}, McMullen defined the Teichmuller polynomial $\Theta_F\in \mathbb{Z}[H_1(N;\mathbb{Z})/Tor]$, which is in the form of
$\Theta_F=\sum_g a_g\cdot g$ with $a_g\in \mathbb{Z},\ g\in H_1(N;\mathbb{Z})/Tor$. Using Teichmuller polynomial $\Theta_F$, one can compute $\lambda(\cdot)$ effectively: $\lambda(\alpha)$ is the largest root of polynomial $\sum_g a_g\cdot X^{\langle \alpha,g\rangle}=0$ for any $\alpha \in C$. Using properties of Teichmuller polynomial, McMullen reproved Fried's and Matsumoto's theorem in \cite{McM}. We will briefly review McMullen's work in Section \ref{McMullen}.

Knowing the existence and uniqueness of the minimal point $m_F$, McMullen asked the following question in \cite{McM}:\\
{\bf Question.} Is the minimum always achieved at a rational cohomology class?

In this paper, we will construct a few examples to show that the minimal point need not to be a rational point, and such irrational minimal point is a general phenomenon.

\subsection{Main Results}
In this paper, we assume all manifolds are oriented, all homeomorphisms preserve the fixed orientation. For a surface $S$ in $3$-manifold $N$, we may abuse notation and use $[S]$ to denote the element in $H^1(N;\mathbb{Z})$ dual with $[S, \partial S]\in H_2(N,\partial N;\mathbb{Z})$. We may also use $c$ to denote the homology class of oriented curve $c$ in $N$.

For a pseudo-Anosov map $\phi$ on surface $S$, we will define a finitely generated $\mathbb{Q}$-submodule $A(S,\phi)$ of $\mathbb{R}$.
It is generated by the coordinate of $m_F$. More specifically, if $\{\alpha_i\}$ is a basis of $H^1(N;\mathbb{Z})$, and
$m_F=\sum r_i \alpha_i$, then we define a $\mathbb{Q}$-submodule of $\mathbb{R}$: $$A(S,\phi)=\{\sum q_ir_i\ |\ q_i\in \mathbb{Q}\}.$$
Since $m_F$ lies on fibered face $F$ which is dual to an integer homology class, $\mathbb{Q}$ is always a submodule of
$A(S,\phi)$. 

Essentially, $A(S,\phi)$ is an invariant for the fibered cone $C$, sometimes we will also use $A_C$ to
denote $A(S,\phi)$. However, in most part of this paper, we will only think $A(S,\phi)$ as an invariant of pseudo-Anosov map.

In Section \ref{property}, we will show a few nice properties of $A(S,\phi)$. At first, in Section \ref{coversec}, we show that $A(S,\phi)$ behaves well under taking finite cover.
\begin{prop}\label{nice}
For two pseudo-Anosov maps $(S_1,\phi_1)$ and $(S_2,\phi_2)$, if there is another manifold $M$ and finite covers $p_i:M\rightarrow M(S_i,\phi_i),\ i=1,2$, such that $p_1^*([S_1])$ and $p_2^*([S_2])$ lie in the same fibered cone of $H^1(M;\mathbb{R})$, then $A(S_1,\phi_1)=A(S_2,\phi_2)$.
\end{prop}

McMullen's question is equivalent to: whether $A(S,\phi)=\mathbb{Q}$ hold for all pseudo-Anosov maps. In Section \ref{specialsec}, we will show that all pseudo-Anosov maps on relatively simple surfaces: $\Sigma_{0,4},\Sigma_{1,2},\Sigma_{2,0}$, always have $A(S,\phi)=\mathbb{Q}$ (here $\Sigma_{g,n}$ denote orientable surface of genus $g$ with $n$ boundary components). On the other hand, in Section \ref{numerical}, \ref{numerical2} and \ref{numerical3}, we will give examples to show that for slightly more complicated surfaces: $\Sigma_{2,1},\Sigma_{3,0},\Sigma_{1,3}$ and $\Sigma_{0,5}$, all these surfaces admit pseudo-Anosov maps with $A(S,\phi)\ne\mathbb{Q}$.

Since different integer classes in the same fibered cone $C$ share the same invariant $A_C$, we can get a
lot of different pseudo-Anosov maps with $A(S,\phi)\ne\mathbb{Q}$ if we are given one. Using the examples we have in hand, we can deduce
the following theorem:
\begin{thm}\label{general}
For surfaces $\Sigma_{g,0}$ with $g\geq 3$ and punctured surfaces $\Sigma_{g,1}$ with $g\geq 2$, all these surfaces admit pseudo-Anosov maps with
$A(S,\phi)\ne\mathbb{Q}$.
\end{thm}

Actually, the author believes that for all surfaces $S$ with negative Euler characteristic and $S \ne \Sigma_{0,3}, \Sigma_{0,4}, \Sigma_{1,1},
\Sigma_{1,2}, \Sigma_{2,0}$, i.e. $|\chi(S)|\geq 3$, $S$ admits a pseudo-Anosov map $\phi$ such that $A(S,\phi)\ne \mathbb{Q}$. However, the author do not have enough
examples to produce pseudo-Anosov maps with irrational minimal point on all these surfaces.

In Section \ref{transsec}, we will show that the coordinate $r_i$ will be in the form of $\log(\alpha)/\log(\beta)$, here $\alpha, \beta$ are algebraic numbers. By an equivalent
formulation of Hilbert's seventh problem, $\log(\alpha)/\log(\beta)$ is either a rational or a transcendental number (\cite{FN}  Theorem 3.2).
So we call $A(S,\phi)$ a transcendental invariant.

In this paper, we will study two different constructions to produce pseudo-Anosov maps which may have $A(S,\phi)\ne \mathbb{Q}$.

{\bf First Construction.} We will introduce the first construction in Section \ref{cons1}, it is given by drilling (or cyclic branched covering) along a closed orbit of the suspension flow. Suppose $N=M(S,\phi)$ is a closed hyperbolic surface bundle, $c$ is a oriented closed orbit of the suspension flow and intersect with $S$ positively. We drill the manifold $N$ along $c$ to get another hyperbolic surface
bundle $N\setminus c$ with natural inclusion $i:\ N\setminus c \rightarrow N$. For the dilatation function, $\lambda(\alpha)=\lambda(i^*(\alpha))$
for $\alpha$ lie in fibered cone $C$. However, for the Thurston norm, $\| i^*(\alpha) \|=\| \alpha \|+\langle \alpha, c \rangle$. The difference of $\| i^*(\alpha) \|$ from $\| \alpha \|$ depends on the cohomology
class $\alpha$, which may produce irrationality.

The fibered cone $C$ corresponds to a homology class $x\in H_1(N;\mathbb{Z})/Tor$, such that $\| \alpha \|= \\ \langle \alpha, x \rangle$ for any $\alpha \in C$. We define
two oriented closed orbits $c_1$ and $c_2$ to be drilling equivalent if $[c_1]+x$ is linear dependant with $[c_2]+x$ in $H_1(N;\mathbb{Z})/Tor$. Then we have the following Drilling Theorem.

\begin{thm}\label{dr}(Drilling Theorem)
For all but finitely many drilling equivalence classes of closed orbit, the drilling construction give irrational minimal point.
\end{thm}

There is also an analogous construction by using cyclic branched cover along closed orbits and a similar theorem (Theorem \ref{branch}). However, due to some technical reason, the statement there is not as neat as in Theorem \ref{dr}.

In Lemma \ref{infty}, we will show that, for any pseudo-Anosov map $(S,\phi)$, the mapping torus $M(S,\phi)$ has infinitely many different drilling classes and also infinitely many branched covering classes satisfying the technical condition in Theorem \ref{branch}. So we have irrational minimal point examples for both closed surfaces and surfaces with boundary.

In Section \ref{example}, we will study some simple drilling class and branched covering class for an explicit example. An alternative method of proving irrationality of $A(S,\phi)$ will be developed there. This method uses some algebraic number theory and numerical computation.

{\bf Second Construction.} The second construction is given in Section \ref{cons2}, which comes from Penner's construction \cite{Pe}. In his construction, Penner took two families of disjoint simple closed curves $\{a_i\}$ and $\{b_j\}$ on surface $S$, such that the union of $\{a_i\}$ and $\{b_j\}$ fill $S$. Let $\mathcal{D}(a^+,b^-)$ be the semigroup generated by positive twists along $a$ curves and negative twists along $b$ curves. Take any $\phi \in \mathcal{D}(a^+,b^-)$ such that each twist along $\{a_i\}$ and $\{b_j\}$ appears at least once in the presentation of $\phi$, then $\phi$ is a pseudo-Anosov map and an invariant bigon track $\tau$ is constructed explicitly. The invariant bigon track $\tau$ only depends on the two families of curves $\{a_i\}$ and $\{b_j\}$, but does not depend on $\phi$. In Proposition \ref{algsymmetry}, we will show that for some special $\phi\in \mathcal{D}(a^+,b^-)$, $A(S,\phi)=\mathbb{Q}$ holds, while $\phi$ does not admit geometric symmetry as in Proposition \ref{symmetry}. On the other hand, in Section \ref{numerical2}, we will give an explicit element in the semigroup not satisfying the condition in Proposition \ref{algsymmetry}, and use the numerical method in Section \ref{example} to show that $A(S,\phi)\ne \mathbb{Q}$.

Being an invariant of fibered cone $C$, $A_C$ is actually not an invariant of $3$-manifold. In Section \ref{globalsec}, we will give an example that is
obtained by Dehn-filling of the magic-manifold. The filled manifold has six fibered faces, one pair of them have $A_C=\mathbb{Q}$, while the other two pairs have $A_C\ne \mathbb{Q}$.

{\bf Acknowledgement:} The author is grateful to his advisor David Gabai for many helpful conversations. The author thanks Shicheng Wang for valuable conversations, Pierre Dehornoy, Eriko Hironaka and Curtis McMullen for comments on a preliminary draft, and Siyu Yang for computational assistance.

\section{McMullen's Work on Teichmuller Polynomial}\label{McMullen}
In this section, we will review McMullen's work on Teichmuller Polynomial. All the material in this section can be found in \cite{McM}.

Given a hyperbolic surface bundle $N$, let $F\subset H^1(N;\mathbb{R})$ be an open fibered face, McMullen defined a polynomial invariant $\Theta_F$ called Teichmuller polynomial of $F$. The Teichmuller polynomial $\Theta_F=\sum_g a_g\cdot g$ lies in the group ring $\mathbb{Z}[G]$, here $G=H_1(N;\mathbb{Z})/Tor$.

For a fibered face $F\subset H^1(N;\mathbb{R})$, it is also associated with a suspended lamination $\mathcal{L}\subset M$, which is the suspension of invariant stable lamination for some surface bundle structure lying in $C=\mathbb{R}^+\cdot F$. Moreover, $\mathcal{L}$ is universal for this fibered face $F$, i.e. it does not depend on which surface bundle structure one choose. Let $p:\ \tilde{M}\rightarrow M$ be the maximal free abelian cover with deck transformation group $G$, and let $\tilde{\mathcal{L}}=p^{-1}(\mathcal{L})$.

The module of lamination $T(\tilde{\mathcal{L}})$ is defined in \cite{McM} Section 2. $T(\tilde{\mathcal{L}})$ is a $\mathbb{Z}$-module generated by transversals $[T]$ of $\tilde{\mathcal{L}}$ modulo equivalent relations:\\
1) $[T]=[T']+[T'']$ if $T$ is disjoint union of $T'$ and $T''$, \\
2) $[T]=[T']$ if $T$ is isotopic with $T'$ with respect to intersection with $\tilde{\mathcal{L}}$.\\
Since $\tilde{\mathcal{L}}$ admits $G$-action by deck transformation, $T[\tilde{\mathcal{L}}]$ is a $\mathbb{Z}[G]$-module. Then $\Theta_F$ is defined to be the Alexander polynomial of $T[\tilde{\mathcal{L}}]$ as a $\mathbb{Z}[G]$-module.

$\Theta_F$ has a few nice properties and we will list some of them in this section.

At first, $\Theta_F$ can be computed effectively if one knows the monodromy $\phi$ on some fiber surface $S$ with $[S]\in C$. Let $\tau\subset S$ be the invariant train track of $\phi$ carrying the stable lamination with switch set $V$ and branch set $E$. Let $\tilde{S}$ be one component of $p^{-1}(S)$ and $\tilde{\tau}\subset \tilde{S}$ be the component of $p^{-1}(\tau)$ lying in $\tilde{S}$. $H_1(M;\mathbb{Z})/Tor$ decomposes as $\mathbb{Z}[u]\oplus T$, here $u\cap [S]=1$ and $t\cap [S]=0$ for any $t\in T$. We will always fix such a decompostion of $H_1(M;\mathbb{Z})/Tor$ in this paper. There is a natural $T$ action on $\tilde{\tau}$, so the module of branches of $\tilde{\tau}$ can be identified with $\mathbb{Z}[T]^E$ and so does the module of switches $\mathbb{Z}[T]^V$. Given a $T$-invariant collapsing $\tilde{\phi}(\tilde{\tau})\rightarrow \tilde{\tau}$, we have a $\mathbb{Z}[T]$-module map $P_E:\mathbb{Z}[T]^E\rightarrow \mathbb{Z}[T]^E$ with matrix $P_E(t)$, and also a map $P_V$ on switches with matrix $P_V(t)$.  Then McMullen showed that:

\begin{thm}\label{formula}(\cite{McM} Theorem 3.6)
The Teichmuller polynomial of the fibered face $F$ is given by: $$\Theta_F(t,u)=\frac{det(uI-P_E(t))}{det(uI-P_V(t))}$$ when $b_1(M)>1$.
\end{thm}

Teichmuller polynomial $\Theta_F$ also has nice symmetric property as Alexander polynomial:

\begin{thm}\label{reverse}(\cite{McM} Corollary 4.3)
The Teichmuller polynomial is symmetric, i.e. $$\Theta_F=\sum_g a_g\cdot g=\pm h \sum_g a_g\cdot g^{-1}$$ for some unit $h\in \mathbb{Z}[G]$.
\end{thm}

In this paper, the most important property of Teichmuller polynomial we will use is, $\Theta_F$ can compute dilatation function $\lambda(\cdot)$ effectively.

\begin{thm}\label{dilatation}(\cite{McM} Theorem 5.1)
The dilatation function $\lambda(\cdot)$ satisfies $$\lambda(\alpha)=sup\{k>1|\ 0=\Theta_F(k^{\alpha})=\sum_g a_g\cdot k^{\langle \alpha,g \rangle}\}$$
for any $\alpha\in C$.
\end{thm}

\begin{rem}\label{PF}
Actually, what MuMullen showed is $\lambda(\alpha)=sup\{|k|\ |\ 0=\Theta_F(k^{\alpha})\}$. Since the supremum is assumed by a positive real number by the Perron-Frobenius theory, we can state the Theorem as above.
\end{rem}

McMullen also defined the Teichmuller norm (with respect to fibered face $F$) on $H^1(N;\mathbb{R})$. For any $\alpha \in H^1(N;\mathbb{R})$, the Teichmuller norm is defined by $\|\alpha\|_{\Theta_F}=\sup_{a_g\ne 0 \ne a_h}\langle \alpha, g-h\rangle $. Then he proved that the Teichmuller norm $\|\cdot\|_{\Theta_F}$ determines the fibered cone $C$.

\begin{thm}\label{cone}(\cite{McM} Theorem 6.1)
For any fibered face $F$ of the Thurston norm unit ball, there exists a face $D$ of the Teichmuller norm unit ball, such that $\mathbb{R}_+\cdot F=\mathbb{R}_+\cdot D$.
\end{thm}

By the formula in Theorem \ref{formula}, we can see that $\Theta_F(u,t)$ has a leading term $u^d$ with coefficient $1$, i.e. $\Theta_F(u,t)=u^d+b_1(t)u^{d-1}+\cdots+b_d(t)$. Then Theorem A.1 (C) of \cite{McM} implies the following immediate corollary:

\begin{col}\label{integer}
For any $\alpha\in C$, $\langle \alpha,d\cdot u\rangle >\langle \alpha,g\rangle$ for any $g$ other than $u^d$ appearing in $\Theta_F$, thus $\sum_g a_g\cdot X^{\langle \alpha,g \rangle}$ has a unique leading term $X^{\langle \alpha,d\cdot u \rangle}$ with coefficient $1$.
\end{col}

\section{Properties of Invariant $A(S,\phi)$}\label{property}

Given a hyperbolic surface bundle $N=M(S,\phi)$, let $C$ be the fibered cone containing the Poincare dual of $[S]$, and let the corresponding fibered face be $F$. Let $m_F$ be the minimal point of the restriction of function
$\lambda(\cdot)$ on the fibered face $F$. After choosing a basis $\{\alpha_i\}_{i=1}^b$ of $H^1(N;\mathbb{Z})$, we have $m_F=\sum_{i=1}^b r_i\alpha_i$. Then we define our invariant to be: $$A_C=A(S,\phi)=\{\sum_{i=1}^b q_i r_i\ |\ q_i\in \mathbb{Q}\}.$$
Since there exists $x\in H_1(N;\mathbb{Z})/Tor$ {\it dual} to fibered face $F$, i.e. $\| \alpha \|=\langle \alpha,x \rangle$ for any $\alpha \in C$, and $1=\|m_F\|=\langle m_F,x\rangle$, $\mathbb{Q} \subset A(S,\phi)$ always holds.

Before giving examples with $A(S,\phi)\ne \mathbb{Q}$, let us first investigate a few nice properties of $A(S,\phi)$ in this section.

\subsection{Covering Property of $A(S,\phi)$}\label{coversec}

Let $p: \tilde{N}\rightarrow N$ be a finite cover, then $p^*([S])$ gives a surface bundle structure on $\tilde{N}$. Let the
fibered cone containing $[S]$ and $p^*([S])$ be $C$ and $C'$, and the corresponding fibered face be $F$ and $F'$ respectively.
Let $S'$ be one component of $p^{-1}(S)$, and $\phi'$ be the corresponding monodromy. Then we have the following proposition:

\begin{prop}\label{inv}
$A(S,\phi)=A(S',\phi')$
\end{prop}

We begin with showing the proposition for regular cover:
\begin{lem}\label{cover}
Suppose $p: \tilde{N}\rightarrow N$ is a regular cover, then $A(S,\phi)=A(S',\phi')$.
\end{lem}

\begin{proof}

Let $H$ be the deck transformation group of regular cover $p: \tilde{N}\rightarrow N$. Then $p^*(H^1(N;\mathbb{R}))$ is the fixed point
set of the $H$ action, i.e. $p^*(H^1(N;\mathbb{R}))=(H^1(\tilde{N};\mathbb{R}))^H$.

In \cite{Ga}, Gabai showed that for any $\alpha \in H^1(N;\mathbb{R})$, $\|p^*(\alpha) \| =\deg{p} \cdot \| \alpha \|$. So $\frac{1}{\deg{p}}p^*(F)\subset F'$. Actually $\frac{1}{\deg{p}}p^*(F)\subset (F')^H$ (the fixed point set of $H$ action on $F'$), since $p^*(H^1(N;\mathbb{R}))=(H^1(\tilde{N};\mathbb{R}))^H$.

Claim. $\frac{1}{\deg{p}}p^*(F)=(F')^H$.

Let $x'\in H_1(\tilde{N};\mathbb{Z})/Tor$ be the homology class dual with $C'$, i.e. $\| \alpha' \|=\langle \alpha',x'\rangle$ for any $\alpha' \in C'$. Then $\frac{1}{\deg{p}}p_*(x')$ is dual to $C$, since
$\| \alpha \|=\frac{1}{\deg{p}}\| p^*(\alpha) \|=\frac{\langle p^*(\alpha),x'\rangle}{\deg{p}}=\frac{\langle \alpha,p_*(x')\rangle}{\deg{p}}$ for any $\alpha \in C$. For any $\beta' \in (F')^H$, let
$\beta'=p^*(\beta)$. Then $\| \beta \|=\frac{1}{\deg{p}}\| p^*(\beta) \|=\frac{1}{\deg{p}}\langle p^*(\beta),x'\rangle =\langle \beta,\frac{1}{\deg{p}}p_*(x')\rangle$. So
$\beta \in C$, thus $\beta' \in p^*(C)$. Since $\|\beta'\|=1$, $\beta' \in \frac{1}{\deg{p}}p^*(F)$ holds immediately.

Since the dilatation function $\lambda(\cdot)$ is invariant under $H$ action, i.e. $\lambda(\alpha')=\lambda(h^*(\alpha'))$ for $\alpha' \in C'$, the minimal point $m_F' \in (F')^H$. So it suffices to find the minimal point of the restriction of $\lambda(\cdot)$ on $(F')^H=\frac{1}{\deg{p}}p^*(F)$.

On the other hand, since $\lambda(p^*(\alpha))=\lambda(\alpha)$ holds for integer class $\alpha\in C$, this equality holds for any $\alpha \in C$. So $\frac{1}{\deg{p}}p^*(m_F)$ is the minimal point of the restriction of $\lambda(\cdot)$ on $\frac{1}{\deg{p}}p^*(F)$, i.e.
$\frac{1}{\deg{p}}p^*(m_F)=m_F'$.

Since $p^*$ is represented by an integer matrix under integer basis of $H^1(N;\mathbb{Z})$ and $H^1(\tilde{N};\mathbb{Z})$, the coordinates of
$m_F$ and $m_F'$ give the same $\mathbb{Q}$-module, i.e. $A(S,\phi)=A(S',\phi')$.
\end{proof}

{\it Proof of Proposition \ref{inv}: } We can take a further finite cover $p':\tilde{\tilde{N}}\rightarrow \tilde {N}$ such that $p'':\tilde{\tilde{N}} \rightarrow N$ is a regular cover. Let $C''$, $F''$ and $m_F''$ be the corresponding fibered cone, fibered face of $N''$ and minimal point on $F''$.

By Lemma \ref{cover}, we have $m_F''=\frac{1}{\deg{p''}}p''^*(m_F)$. Since
$$m_F''=\frac{1}{\deg{p''}}p''^*(m_F)=\frac{1}{\deg{p'}}p'^*(\frac{1}{\deg{p}}p^*(m_F))$$ and $\frac{1}{\deg{p'}}p'^*(F')\subset F''$,
$\frac{1}{\deg{p}}p^*(m_F)$ is the minimal point
on the fibered face $F'$, i.e. $m_F'=\frac{1}{\deg{p}}p^*(m_F)$. So $A(S,\phi)=A(S',\phi')$. \hfill $\qed$

Comparing with definitions of commensurability in \cite{CSW}, we give the following definition:
\begin{defn}
Two maps $(S_1,\phi_1)$ and $(S_2,\phi_2)$ are said to be {\it fibered cone commensurable} if there is another manifold $M$, with finite covers $p_i:M\rightarrow M(S_i,\phi_i),\ i=1,2$, such that $p_1^*([S_1])$ and $p_2^*([S_2])$ lie in the same fibered cone of $H^1(M;\mathbb{R})$.
\end{defn}

We have the following immediate Corollary of Proposition \ref{inv}, which rephrases Proposition \ref{nice}.
\begin{col}\label{inv2}
If two pseudo-Anosov maps $(S_1,\phi_1)$ and $(S_2,\phi_2)$ are fibered cone commensurable, then $A(S_1,\phi_1)=A(S_2,\phi_2)$.
\end{col}

\subsection{Symmetry Implies Rationality}\label{specialsec}

Using the symmetry from group action, we can deduce $A(S,\phi)=\mathbb{Q}$ in a few simple cases.

\begin{prop}\label{symmetry}
Suppose $\phi$ is a pseudo-Anosov map on surface $S$, and $\phi$ commutes with an involution $\tau$ with $\tau_*=-id$ on $H_1(S;\mathbb{Z})$.
Then $A(S,\phi)=\mathbb{Q}$.
\end{prop}

\begin{proof}	
On $3$-manifold $N=M(S,\phi)=S\times I/(x,0) \sim (\phi(x),1)$, we can define involution $\bar{\tau}$ on $N$ by $\bar{\tau}(x,t)=(\tau(x),t)$. $\bar{\tau}$ is
well-define since $\phi$ commutes with $\tau$.

Let $\pi:N\rightarrow S^1$ gives the surface bundle structure of $M(S,\phi)$ with fiber $S$. Let $(t_1,t_2,\cdots,t_k,u)$ be a basis of
$H_1(N;\mathbb{Z})/Tor$, such that $\pi_*(t_i)=0$. Then $\bar{\tau}_*(t_i)=-t_i$, choose $u$ such that $\bar{\tau}_*(u)=u$.
Let $(\alpha_1,\cdots,\alpha_k,[S])$ be the
dual basis in $H^1(N;\mathbb{Z})$, then $\bar{\tau}^*(\alpha_i)=-\alpha_i$, while $\bar{\tau}^*([S])=[S]$.

Let $C$ be the fibered cone containing $[S]$ and $F$ be the corresponding fibered face. Since $\frac{[S]}{\|[S]\|}$ is the unique fixed point
of the $\bar{\tau}^*$ action on $F$, and $\|\frac{[S]}{\|[S]\|}\|=1$, we have $m_F=\frac{[S]}{\|[S]\|}$, which is a rational class. So $A(S,\phi)=\mathbb{Q}$.
\end{proof}

Pseudo-Anosov maps which commute with the hyperelliptic involution on closed surfaces are closely related with pseudo-Anosov braids, which are specialized interesting. So we point out the following immediately corollary.
\begin{col}
For any pseudo-Anosov map $\phi$ on surface $S$ which commutes with the hyperelliptic involution $\tau$, we have $A(S,\phi)=\mathbb{Q}$.
\end{col}

\begin{col}\label{special}
For any pseudo-Anosov map $\phi$ on closed surface $S=\Sigma_{2,0}$ or $\Sigma_{1,2}$ or $\Sigma_{0,4}$, $A(S,\phi)=\mathbb{Q}$.
\end{col}

\begin{proof}
By Proposition \ref{inv}, $A(S,\phi)$ is invariant by taking powers of $\phi$, so we assume $\phi$ lies in the pure mapping class group,
i.e. the mapping classes send each boundary component of the surface to itself.

By \cite{FM} Section 4.4.4, the pure mapping class groups of $\Sigma_{2,0}, \Sigma_{1,2}$ are generated by Dehn twists along simple
closed curves $\gamma_1,\ \gamma_2,\ \gamma_3$ in Figure 1 (a), (b). It is also well known that the pure mapping class group of $\Sigma_{0,4}$ is generated by twists along simple closed curves $\gamma_1,\ \gamma_2$
in Figure 1 (c).

It is easy to see that the involutions $\tau$ ($\pi$ rotation) in Figure 1 (a), (b) commute with the whole pure mapping class group for $\Sigma_{2,0}$ and
$\Sigma_{1,2}$ and $\tau_*=-id$ on $H_1(S;\mathbb{R})$. So by Proposition \ref{symmetry}, $A(S,\phi)=\mathbb{Q}$.

Although $\Sigma_{0,4}$ does not admit an involution with $\tau_*=-id$, it admits two involutions $\tau_1$ and $\tau_2$ both commute with the pure
mapping class group (see Figure 1 (c)). These two involutions give a $\mathbb{Z}_2 \oplus \mathbb{Z}_2$ action on $N=M(S,\phi)$, and an action on fibered face $F$. The minimal point $m_F$ is the unique fixed point of this
$\mathbb{Z}_2\oplus \mathbb{Z}_2$ action on $F$. So $A(S,\phi)=\mathbb{Q}$.

\end{proof}

\begin{center}
\psfrag{a}[]{$\gamma_3$} \psfrag{b}[]{$\gamma_2$} \psfrag{c}[]{$\gamma_1$}
\psfrag{d}[]{$\tau$} \psfrag{e}[]{$(a)$} \psfrag{f}[]{$\gamma_2$}
\psfrag{g}[]{$\gamma_3$} \psfrag{h}[]{$\gamma_1$} \psfrag{i}[]{$\tau$}
\psfrag{j}[]{$(b)$} \psfrag{k}[]{$\gamma_1$} \psfrag{l}[]{$\gamma_2$}
\psfrag{m}[]{$\tau_1$} \psfrag{n}[]{$\tau_2$} \psfrag{o}[]{$(c)$}
\includegraphics[width=5.5in]{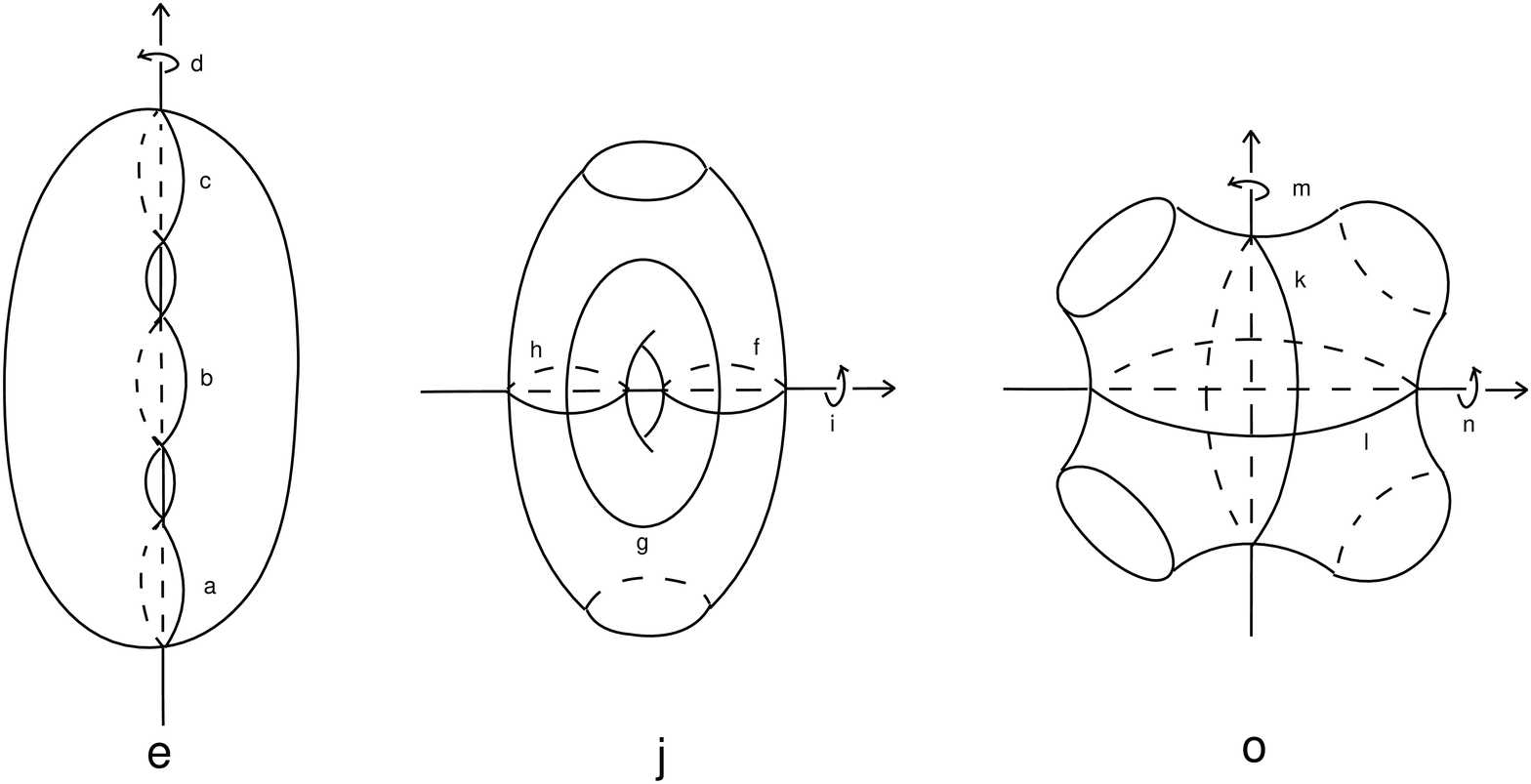}
\vskip 0.5 truecm
 \centerline{Figure 1}
\end{center}

A quick but interesting corollary of Corollary \ref{inv2} and Corollary \ref{special} is the following:
\begin{col}
For closed surface $S$, if $A(S,\phi)\ne \mathbb{Q}$,\\
(a) $(S,\phi)$ is not fibered cone commensurable with any pseudo-Anosov map on $\Sigma_{2,0}$;\\
(b) for any pseudo-Anosov map $(S',\phi')$ with $b_1(M(S',\phi'))=1$, $(S,\phi)$ is not fibered cone commensurable with $(S',\phi')$.
\end{col}

\subsection{Irrationality Implies Transcendentality}\label{transsec}

The following theorem in number theory is an equivalent formulation of Hilbert's Seventh Problem (which explains why logarithm function is called
{\it transcendental function}):
\begin{thm}\label{Hilbert}(\cite{FN} Theorem 3.2)
 Let $\alpha,\beta,\gamma$ be algebraic numbers, and $\alpha \beta \ln{\beta} \ne 0$, with $\gamma=\frac{\ln{\alpha}}{\ln{\beta}}$, then $\gamma \in \mathbb{Q}$.
\end{thm}

This theorem implies that all the irrational elements of $A(S,\phi)$ are transcendental:
\begin{prop}\label{transcendental}
 Suppose that $A(S,\phi)\ne \mathbb{Q}$, then for any $\omega \in A(S,\phi)\setminus \mathbb{Q}$, $\omega$ is a transcendental number.
\end{prop}
\begin{proof}
Take a basis $(\alpha_1,\cdots,\alpha_b)$ of $H^1(N;\mathbb{Z})$, and a dual basis $(x_1,\cdots,x_b)$ of $H_1(N;\mathbb{Z})/Tor$ (here $b=b_1(M)$). Also take a basis
$(v_1,\cdots,v_{b-1})$ of the tangent space $T_{m_F}F$. Let the Teichmuller polynomial $\Theta_F=\sum_ga_g\cdot g$, and let
each $g$ be expressed as $g=\sum_{i=1}^bg_ix_i$, here $g_i\in \mathbb{Z}$.

Let $m_F=\sum_{i=1}^b r_i\alpha_i$ be the minimal point of the restriction of $\lambda(\cdot)$ on fibered face $F$. Since $\|m_F\|=1$,
for $x=\sum_{i=1}^b m_ix_i\in H_1(N;\mathbb{Z})/Tor$ dual to fibered face $F$, we have $\|m_F\|=\langle m_F,x\rangle=\sum_{i=1}^b r_i\cdot m_i=1$,
where $m_i\in \mathbb{Z}$.

Let $\lambda_0=\lambda(m_F)$. Then by Theorem \ref{dilatation},
$$0=\sum_g a_g\lambda_0^{\langle m_F,g \rangle}=\sum_g (a_g \prod_{i=1}^b(\lambda_0^{r_i})^{g_i}).$$
Since $m_F$ is the minimal point, by taking derivatives along direction $v_i,i=1,\cdots,b-1$, we get
$$0=\sum_g a_g \langle v_i,g \rangle \lambda_0^{\langle m_F,g \rangle}=\sum_g (a_g  \langle v_i,g \rangle\prod_{i=1}^b(\lambda_0^{r_i})^{g_i}),\ i=1,\cdots,b-1.$$

So $(\lambda_0^{r_1},\cdots, \lambda_0^{r_b})$ is a solution of polynomial equation:
$$
\left\{ \begin{array}{l}
        \sum_g a_gX_1^{g_1}\cdots X_b^{g_b}=0, \\
        \sum_g a_g\langle v_1,g\rangle X_1^{g_1}\cdots X_b^{g_b}=0, \\
        \cdots \\
        \sum_g a_g \langle v_{b-1},g \rangle X_1^{g_1}\cdots X_b^{g_b}=0.
\end{array} \right.
$$

Since the minimal point $m_F$ is unique, these equations are independent, and one of the solution is $(\lambda_0^{r_1},\cdots, \lambda_0^{r_b})$.
Solve the equation by eliminating free variables inductively, we know that $\lambda_0^{r_1},\cdots, \lambda_0^{r_b}$ are all algebraic numbers. For any $r \in A(S,\phi)$, if $r=\sum_{i=1}^b q_i\cdot r_i$
with $q_i\in \mathbb{Q}$, then $$r=\frac{\ln{(\prod_{i=1}^b (\lambda_0^{r_i})^{q_i})}}{\ln{\lambda_0}}= \frac{\ln{(\prod_{i=1}^b (\lambda_0^{r_i})^{q_i})}}
{\ln{(\prod_{i=1}^b (\lambda_0^{r_i})^{m_i})}}.$$

Since $\lambda_0=\lambda(m_F)$ is a positive real number, $\lambda_0^{r_i}$ are all positive real algebraic numbers. For $q_i,m_i \in \mathbb{Q}$, $\prod_{i=1}^b (\lambda_0^{n_i})^{q_i}$ and
$\prod_{i=1}^b (\lambda_0^{n_i})^{m_i}$ are both positive real algebraic numbers. By Theorem \ref{Hilbert}, $r$ is either a
rational number of a transcendental number.
\end{proof}

Given Proposition \ref{transcendental}, it makes sense that we call $A(S,\phi)$ a {\it transcendental invariant} of pseudo-Anosov maps.

\section{Drilling and Branched Covering Theorem}\label{cons1}

\subsection{Drilling Theorem}

Given a pseudo-Anosov map $\phi$ on closed surface $S$, it is associated with a hyperbolic surface bundle
$N=M(S,\phi)=S\times I /(x,0)\sim (\phi(x),1)$ and a suspension flow on $N$, while this flow is universal for the fibered face. We will call a closed orbit of the suspension flow a {\it primitive closed orbit} if it goes around a circle only once, and we sometimes only call it a closed orbit. If we want to talk about a closed orbit that goes around a circle more than once, we will call it a {\it nonprimitive closed orbit}.

 Take a primitive closed orbit $c\subset N$ of the suspension flow on
$N=M(S,\phi)$. Let the manifold given by drilling $N$ along $c$ be $N_c=N\setminus c$, which is also a hyperbolic surface bundle.
Let $S_c=S\setminus (S\cap c)$, and $\phi_c=\phi|_{S_c}$,
then we have $N_c=M(S_c,\phi_c)$. There is a natural inclusion
$i_c:\ N_c\rightarrow N$.

\begin{lem}\label{iso}
$i_c^*:H^1(N;\mathbb{R})\rightarrow H^1(N_c;\mathbb{R})$ is an isomorphism.
\end{lem}
\begin{proof}
In the proof of this Lemma, all the (co)homology groups have $\mathbb{R}$-coefficients, and we will omit the coefficient.

Since $N=N_c\cup_{T^2}D^2\times S^1$, we have M-V sequence
$$0\rightarrow H^1(N)\rightarrow H^1(N_c)\oplus H^1(D^2\times S^1)\rightarrow H^1(T^2)\rightarrow \cdots.$$

Since $H^1(D^2\times S^1)\rightarrow H^1(T^2)$ is injective, $i_c^*:H^1(N)\rightarrow H^1(N_c)$ is injective. So it suffice to show
$H^1(N)$ and $H^1(N_c)$ have the same dimension. By Lefschetz duality, we need only to show that $H_2(N)$ and $H_2(N_c,\partial N_c)$ have
the same dimension.

By excision, we have $H_2(N_c,\partial N_c)\cong H_2(N,c)$. By M-V sequence
$$0\rightarrow H_2(N)\rightarrow H_2(N,c) \rightarrow H_1(c)\rightarrow H_1(N)\rightarrow \cdots,$$
and $H_1(c)\rightarrow H_1(N)$ is injective, $H_2(N)\cong H_2(N,c)\cong H_2(N_c,\partial N_c)$.
\end{proof}

In general, if $S$ is not closed, $dim(H^1(N_c))$ maybe greater than $dim(H^1(N))$.

Let $C\subset H^1(N;\mathbb{R})$ and $C'\subset H^1(N_c;\mathbb{R})$ be the fibered cone containing the dual of $[S]$ and $[S_c]$,
while $F$ and $F'$ be the corresponding fibered face respectively. 

For any $\alpha \in C$, it is easy to see that $\lambda_C(\alpha)=\lambda_{C'}(i_c^*(\alpha))$, since it holds for integer classes. Since $\lambda(\alpha)$ goes to infinity when $\alpha$ goes to $\partial C$, we have $i_c^*(C)=C'$.
Let $x \in H_1(N;\mathbb{Z})/Tor$ be the dual of fibered cone $C$, i.e. for any $\alpha \in C$, $\|\alpha\|=\langle \alpha,x \rangle$ holds. Choose an orientation on $c$ such that $c$ intersects $[S]$ positively, and let the homology class of $c$ also denoted by $c$. Then
$\|i_c^*(\alpha)\|=\|\alpha\|+\langle \alpha,c \rangle=\langle \alpha,c+x \rangle$ for any $\alpha \in C$. This equality implies that although $i_c^*(C)=C'$, $i_c^*(F)$ may not be parallel with $F'$.

Let $F_c$ denote $(i_c^*)^{-1}(F')\subset H^1(N;\mathbb{R})$ and $m_{F_c}$ be the minimal point of
the restriction of $\lambda_C(\cdot)$ on $F_c$, then we have $i_c^*(m_{F_c})=m_{F'}$.
So to compute $A(S',\phi')$, we need only to compute the coordinates of $m_{F_c}$.
Actually, if $x$ and $c$ are linear dependent, $F_c$ is parallel to $F$, thus $m_{F_c}$ is a scaling of $m_F$.
Otherwise, $F_c$ is tilted with respect to $F$, and the number theoretical property of $m_{F_c}$ and $m_F$ can be quite different.

Let $c_1$ and $c_2$ be two different oriented closed orbits of suspension flow, and both of them intersect $[S]$ positively. We say $c_1$ and $c_2$
are {\it drilling equivalent} if $x+c_1$ is linear dependent with $x+c_2$.
 In this case, $F_{c_1}$ is parallel
with $F_{c_2}$, so $m_{F_{c_1}}$ is a scaling of $m_{F_{c_2}}$ and $A(S_{c_1},\phi_{c_1})=A(S_{c_2},\phi_{c_2})$.
Then we have the following {\it Drilling Theorem}:
\begin{thm}\label{drill}
 For all but finitely many drilling equivalent classes $c$, $A(S_c,\phi_c)\ne \mathbb{Q}$.
\end{thm}
\begin{proof}
 For a drilling class $c$, we have $F_c=\{\alpha \in C|\ \langle \alpha,c+x \rangle=1\}$. Let $m_{F_c}$ be the minimal point of the restriction of $\lambda(\cdot)$ on $F_c$.

Let Teichmuller polynomial $\Theta_F=\sum_g a_g\cdot g$, then by Theorem \ref{dilatation},
$$\sum_g a_g\lambda(m_{F_c})^{\langle m_{F_c},g\rangle}=0.$$ Let $b=b_1(N)$, take a basis
$(v_1,\cdots,v_{b-1})$ of tangent plane $T_{m_{F_c}}F_c$. Since $m_{F_c}$ is the minimal point,
by taking derivative along direction $v_i$, we also have equations $$\sum_g a_g\langle v_i,g\rangle \lambda(m_{F_c})^{\langle m_{F_c},g\rangle }=0,\  i=1,\cdots, b-1.$$

Let $(\alpha_1,\cdots,\alpha_b)$ be a basis of $H^1(N;\mathbb{Z})$, and $(x_1,\cdots,x_b)$ be the dual basis of $H_1(N;\mathbb{Z})/Tor$.
Suppose $m_{F_c}$ is a rational class, then $m_{F_c}=\sum_{i=1}^b \frac{p_i}{q_c}\alpha_i$, here $p_i,q_c \in \mathbb{Z}$ and $gcd(p_1,\cdots,p_b,q_c)$ $=1$. Since $\langle m_{F_c},x\rangle=1$, $gcd(p_1,\cdots,p_b)=1$ holds. For each $g$ appears in the Teichmuller polynomial, let $g=\sum_{i=1}^{b}g_ix_i$, here $g_i \in \mathbb{Z}$.

Then $(\lambda(m_{F_c})^{\frac{p_1}{q_c}},\cdots,\lambda(m_{F_c})^{\frac{p_b}{q_c}})$ is a solution of equation:
$$
\left\{ \begin{array}{l}
        \sum a_gX_1^{g_1}\cdots X_b^{g_b}=0, \\
        \sum a_g\langle v_1,g \rangle X_1^{g_1}\cdots X_b^{g_b}=0, \\
        \cdots \\
        \sum a_g \langle v_{b-1},g \rangle X_1^{g_1}\cdots X_b^{g_b}=0.
\end{array} \right.
$$
Here only $v_1,\cdots,v_{b-1}$ depend on the drilling class $c$, thus only the coefficients of the equation depend on $c$, the degrees do not. By
solving the equation by eliminating free variables inductively, we know number field
$\mathbb{F}=\mathbb{Q}(\lambda(m_{F_c})^{\frac{p_1}{q_c}},\cdots,\lambda(m_{F_c})^{\frac{p_b}{q_c}})$
is a finite extension over $\mathbb{Q}$ with $[\mathbb{F}:\mathbb{Q}]\leq D$ for some $D$. Here $D$ only depends on $\Theta_F$, but does not depend on $c$. Since $gcd(p_1,\cdots,p_b)=1$,
we have $\mathbb{F}=\mathbb{Q}(\lambda(m_{F_c})^{\frac{1}{q_c}})$, so $\deg(\lambda(m_{F_c})^{\frac{1}{q_c}})\leq D$.

Since $\lambda_c=\lambda(m_{F_c})^{\frac{1}{q_c}}$ is the largest root of the Teichmuller polynomial $\Theta_F$ at $(p_1,\cdots,p_b)$
(Remark \ref{PF}), i.e. it is the largest root of $\sum a_g X^{\langle q_cm_{F_c},g\rangle}=0$, all the algebraic conjugations of
$\lambda_c$ has modulus smaller or equal to $\lambda_c$. On the other hand, by Corollary \ref{integer}, $\sum_g a_gX^{\langle q_cm_{F_c},g \rangle}$ has a unique leading term $X^{\langle q_cm_{F_c},d\cdot u\rangle}$ with coefficient $1$ while all the terms have integer coefficients and integer powers. So $\lambda_c\leq \sum_g |a_g|=D'$, and $\lambda_c$ is an algebraic integer.

Given $D,D'\in \mathbb{Z}_+$, there are only finitely many algebraic integers $\lambda$, such that $|\lambda|\leq D'$, $\deg{(\lambda)}\leq D$,
and $\lambda$ has the greatest modulus among its algebraic conjugations. This is because the minimal polynomial of $\lambda$ has bounded degree and
the coefficients are also bounded.

So if there are infinitely many drilling classes $c$ with $A(S_c,\phi_c)=\mathbb{Q}$, then there are infinitely many drilling classes
$c_1, c_2,\cdots$ such that $\lambda(m_{F_{c_i}})^{\frac{1}{q_{c_i}}}=\lambda_{c_i}=\lambda_0$. Here $m_{F_{c_i}}$ is the minimal point of $\lambda(\cdot)$ on $F_{c_i}$,
$q_{c_i}m_{F_{c_i}}$ is an integer class with $\lambda(q_{c_i}m_{F_{c_i}})=\lambda_0$. Since $q_{c_i}F_{c_i}$ is the tangent plane of
hypersurface $\{\alpha\in C|\ \lambda(\alpha)=\lambda_0\}$ with tangent point $q_{c_i}m_{F_{c_i}}$, different drilling class $c_i$ corresponds to different
integer point $q_{c_i}m_{F_{c_i}}$. So there are infinitely many integer points $q_{c_i}m_{F_{c_i}}$ on the
hypersurface $\lambda(\alpha)=\lambda_0$.

However, we claim that for any fixed $\lambda_0\in \mathbb{R}_+$, there are at most finitely many integer points on the hypersurface
$\lambda(\alpha)=\lambda_0$ which correspond to $q_cm_{F_c}$ for some drilling class $c$, thus get a contradiction.

{\it Proof of Claim:}
Suppose there are infinitely many integer points $\alpha\in C$ with $\lambda(\alpha)=\lambda_0$ and correspond to $q_cm_{F_c}$.

Let the fibered cone $C$ be equal to $\{x\in H^1(N,\mathbb{R})|\ p_i(x)>0,i=1,\cdots,n\}$, here each $p_i$ is an integer homology class in $H^1(N,\mathbb{Z})/Tor$. Take an
integer point $\beta \in C$ such that $\lambda(\beta)=\lambda_0$, let $P_i=p_i(\beta)$. Let $C'=\{\alpha \in C| p_i(\alpha)\geq P_i, i=1,\cdots,n\}$.

Suppose there is an integer point $\beta' \in C'$, $\beta'\ne \beta$ such that $\lambda(\beta')=\lambda_0$, and $\beta'=q_cm_{F_c}$ for some drilling class $c$.
Since $\beta' \in C'$, we have $p_i(\beta'-\beta)=p_i(\beta')-P_i\geq 0$. Let $\bar{C}$ be the closure of $C$, then $\bar{C}=\{x\in H^1(N,\mathbb{R})|\ p_i(x)\geq 0,i=1,\cdots,n\}$, we have
$\beta'-\beta \in \bar{C}$. Since $\beta \in C$ and $\beta \ne \beta'$, so $\langle \beta,c+x \rangle >0$, and $\langle \beta'-\beta,c+x \rangle >0$ for any closed orbit $c$.

Since $\beta'=q_cm_{F_c}$, $\beta'$ is the minimal point of the restriction of $\lambda(\cdot)$ on
$q_cF_c=\{\alpha \in C|\langle \alpha, c+x \rangle=\langle \beta', c+x\rangle\}$. However
$\lambda(\frac{\langle \beta', c+x \rangle}{\langle \beta, c+x \rangle}\beta)=\lambda(\beta)^{\frac{\langle \beta, c+x \rangle}
{\langle \beta', c+x \rangle}}<\lambda(\beta)=\lambda(\beta')$,
while $\frac{\langle \beta', c+x \rangle}{\langle \beta, c+x \rangle}\beta \in \{\alpha \in C|\langle \alpha, c+x \rangle=\langle \beta', c+x\rangle\}$.
It contradicts with the fact that $\beta'$ is the minimal point of $\lambda(\cdot)$ on $q_cF_c$.

So for any integer point $\beta'\in C$ satisfying the condition in the claim, we have $p_i(\beta')= j$, for some $i\in \{1,\cdots,n\}$ and
$j\in \{1,\cdots,P_i-1\}$, or $\beta'=\beta$. Since there are infinitely many such integer points, infinitely many of them satisfy $p_{i_0}(\alpha)=j_0$.
Here $\{\alpha \in C|\ p_{i_0}(\alpha)=j_0\}$ is a codimension $1$ hyperplane of $C$.

The argument used above can be repeated inductively to reduce the dimension. Finally we can reduce to the case that there is a ray
$\mathcal{R}\subset C$ which contains infinitely many integer points $\alpha$ with $\lambda(\alpha)=\lambda_0$.
However, since $\frac{1}{\log{\lambda(\cdot)}}$ is either strictly concave or linear (non constant) on rays (\cite{Ma}, \cite{McM}), there are at most two points on $\mathcal{R}$ assuming $\lambda_0$ by function $\lambda(\cdot)$, and we get a contradiction.
\end{proof}

\subsection{Branched Covering Theorem}
An analogy of the drilling construction is the following branched covering construction. Branched covering construction can give pseudo-Anosov maps on closed surfaces with irrational minimal point.

Let $N=M(S,\phi)$ be a closed hyperbolic surface bundle, let $c$ be an oriented closed orbit of the suspension flow, and $d(c)$ be the greatest common divisor of the coordinate of $c$ in $H_1(N;\mathbb{Z})/Tor$. If $gcd(Tor(H_1(N)),d(c))=1$, we can construct a $d(c)$-sheets cyclic branched cover of $N$ along $c$. To make this construction, we will first construct a cyclic cover of $N\setminus c$, then fill in a solid torus.

We have the following long exact sequence (with coefficients in
$\mathbb{Z}$): $$H_2(D^2\times S^1) \oplus H_2(N\setminus c)\rightarrow H_2(N) \rightarrow H_1(T^2)
\rightarrow H_1(D^2\times S^1) \oplus H_1(N\setminus c)\rightarrow H_1(N)\rightarrow 0.$$

For any $\alpha \in H_2(N)$, it is in the image of the first map in the exact sequence if and only if $\alpha \cap c=0$. Since
$H_2(N) \cong H^1(N) \cong Hom(H_1(N),\mathbb{Z})$, we know that for any $\alpha \in H_2(N)$, $d(c)$ is a divisor of $\alpha \cap c$, and there
exists $\alpha_0 \in H_2(N)$ such that $\alpha_0 \cap c=d(c)$. So the image of the second map is $\{kd(c)[m]|\ k\in \mathbb{Z}\}$, here
$[m]$ is the meridian class of the solid torus neighborhood of $c$. So $[m]$ is a torsion element in $H_1(N\setminus c)$ with order $d(c)$.

If $gcd(Tor(H_1(N)),d(c))=1$, since $|Tor(H_1(N\setminus c))|=d(c)\cdot |Tor(H_1(N))|$, $[m]$ generates a direct summand $Tor(H_1(N\setminus c))$ of order $d(c)$. So $Tor(H_1(N\setminus c))=\mathbb{Z}_{d(c)}\oplus A$. Then we can
define a map $H_1(N\setminus c)\rightarrow Tor(H_1(N\setminus c)) \rightarrow \mathbb{Z}_{d(c)}$, which gives a cyclic cover of $N\setminus c$,
thus a cyclic branched cover $p^c:N^c\rightarrow N$ along $c$ of degree $d(c)$. Let $S^c=(p^c) ^{-1}(S)$ (which is connected), then $S^c$ gives a surface bundle structure of $N^c$, and let $\phi^c$ be the corresponding pseudo-Anosov monodromy.

\begin{rem}
We can also get some other cyclic branched covers of $N$ along $c$ even if $gcd(Tor(H_1(N)),d(c))\ne 1$, but the covering degree can not be computed simply by the homology class of $c$.
\end{rem}

As the definition of drilling equivalent class, we will define branched covering equivalent class here. For two different closed orbits of
suspension flow $c_1$ and $c_2$, we say $c_1$ and $c_2$ are {\it branched covering equivalent} if $d(c_1)x+(d(c_1)-1)c_1$ is linear dependent
with $d(c_2)x+(d(c_2)-1)c_2$.

Then we have the following {\it Branched Covering Theorem}
\begin{thm}\label{branch}
 For all but finitely many branched covering classes $c$ satisfying $d(c)>1$ and \\
 $gcd(d(c),Tor(H_1(N;\mathbb{Z})))=1$, $A(S^c,\phi^c)\ne \mathbb{Q}$.
\end{thm}

\begin{proof}
Let $C$ be the fibered cone in $H^1(N;\mathbb{R})$ containing the dual of $[S]$ and $F$ be the corresponding fibered face. $F'$ and $C'$ are defined similarly for $[S^c]$.

Since $p^c:N^c\rightarrow N$ is a $d(c)$-sheet cyclic branched covering, we have an $H=\mathbb{Z}_{d(c)}$ action on $N^c$. As the regular covering case (Lemma \ref{cover}), we have that $H^1(N^c;\mathbb{R})^H=(p^c)^*(H^1(N;\mathbb{R}))$. Since the $H$-action fixes $[S^c]$, $C'$ and $F'$ are both invariant subsets of the $H$ action. Since the minimal point is unique, we have $m_{F'}\in (F')^H$.

 For primitive element $\alpha \in H^1(N;\mathbb{Z})\cap C$ with dual surface $S_{\alpha}$, the dual surface $S_{\alpha}^c$ of
$(p^c)^*(\alpha)$ is a $d(c)$-sheet cyclic branched cover of surface $S_{\alpha}$ and the monodromy on $S_{\alpha}^c$ is a lifting of
monodromy of $S_{\alpha}$. We have $\lambda(p^*(\alpha))=\lambda(\alpha)$, so the equality holds for any $\alpha \in C$.

One the other hand , by Riemann-Hurwitz formula,
$\| (p^c)^*(\alpha)\|=\langle \alpha, (d(c)-1)c+d(c)x\rangle$ for all
$\alpha \in H^1(N;\mathbb{Z})\cap C$. So it also holds for any $\alpha \in C$.

As in the drilling construction, $(C')^H=(p^c)^*(C)$. By considering Thurston norm, $((p^c)^*)^{-1}((F')^G)=\{\alpha \in C|\ \langle \alpha, (d(c)-1)c+d(c)x \rangle=1\}$, and we denote it by $F^c$. Since $\lambda((p^c)^*(\alpha))=\lambda(\alpha)$, the minimal point $m_{F^c}$ of the restriction of $\lambda(\cdot)$ on $F^c$ satisfies $(p^c)^*(m_{F^c})=m_{F'}$. So we need only to show that for all but infinitely many branched covering classes $c$ satisfying the assumption of the theorem, $m_{F^c}$ is irrational. The remaining part of the proof is same with the proof in Theorem \ref{drill}.
\end{proof}

\subsection{Infinitely Many Closed Orbit Classes}

To make the Drilling Theorem and Branched Covering Theorem really gives us some drilling class and branched covering class $c$ which gives irrational invariant, we need to show that, for any pseudo-Anosov map, there are infinitely many different drilling classes and branched covering classes (satisfying the condition in the Branched Covering Theorem). We have the following lemma which helps us to find enough closed orbits of suspension flow. The proof of this Lemma is quite tedious and possibly well-known for experts, but the author can not find a proper literature about it.

\begin{lem}\label{infty}
For any pseudo-Anosov map $\phi$ on surface $S$, \\
(a) there exists infinitely many different drilling classes in $N=M(S,\phi)$;\\
(b) there exists infinitely many different branched covering classes in $N=M(S,\phi)$ satisfy the condition of Branched Covering Theorem.
\end{lem}

\begin{proof}
Let $H_1(N;\mathbb{Z})/Tor=\mathbb{Z}[u]\oplus T$, here $u$ gives the $S^1$-direction and $T$ is the image of $H_1(S;\mathbb{Z})$. This decomposition will give us a coordinate of $H_1(N;\mathbb{Z})/Tor$. Let $x=(n_0,t_0)\in H_1(N;\mathbb{Z})/Tor$ ($n_0>0$) be the homology class dual with fibered face $F$ under the coordinate given about. Let $p:\tilde{N}\rightarrow N$ be the maximal free abelian cover and $\tilde{S}$ be one component of $p^{-1}(S)$.

Take a Markov partition (rectangle partition) $\mathcal{R}=\{R_i\}_{i=1}^k$ of surface $S$, the transition matrix $M_{k\times k}$ is defined by $M_{i,j}=|\phi^{-1}(int(R_j))\cap R_i|$. Then $M$ is a Perron-Frobenius matrix (see \cite{CB} and \cite{FLP} Expose 10), which means that for all $n\geq k$, all the entries of $M^n$ are positive integers. Furthermore, we can lift this Markov partition to a Markov partition on $\tilde{S}$ with $\tilde{\mathcal{R}}=T\cdot \{\tilde{R}_i\}_{i=1}^k$. Also take a lift of $\phi$ correspond to $u\in H_1(N;\mathbb{Z})/Tor$ and denote it by $\tilde{\phi}$. The transition matrix $\tilde{M}_{k\times k}$ is defined by $\tilde{M}_{i,j}=\sum |\tilde{\phi}^{-1}(int(\tilde{R}_j))\cap t\cdot \tilde{R}_i|\cdot t$. Then $\tilde{M}^{m}_{n,n}$ has a nonzero $t^{-1}$ term implies that $N$ has a closed orbit (possibly nonprimitive) with homology class $(m,t)$ which intersects $R_n$. This property allows us to use multiplication of matrix $\tilde{M}$ to study (possibly nonprimitive) closed orbits of $N$.

Let $Cone(\phi)\subset H_1(N;\mathbb{R})$ be the smallest convex closed cone containing all the homology classes of primitive periodic orbits. In \cite{FLP} Expose 14, Fried showed that for any $\alpha\in H^1(N;\mathbb{R})$, $\alpha$ lies in the fibered cone $C$ if and only if $\langle \alpha,x\rangle\geq 0$ for any $x\in Cone(\phi)$. Since we assume $b_1(N)\geq 2$, $\phi$ has two closed orbits $d_1,d_2'$ with linear independent homology classes. Suppose $d_1\cap R_1\ne \emptyset$ and $d_2'\cap R_2 \ne \emptyset$.

Since $\tilde{M}$ is Perron-Fronbenius, all the entries of $\tilde{M}^k$ are nonzero, so $\tilde{M}^k_{1,2},\tilde{M}^k_{2,1}\ne 0$. By precompose and postcompose $(d_2')^t$ with two fixed paths given by $\tilde{M}^N_{1,2},\tilde{M}^N_{2,1}$, we can get a possibly nonprimitive closed orbit $d_2$ which intersects $R_1$ and we can choose $t$ to be large enough such that $d_1$ and $d_2$ are linear independent. We can further take powers of $d_1$ and $d_2$ to make them share the same coefficient on $u$-component. Without changing the symbol, we get two possibly nonprimitive closed orbits $d_1=(m,t_1')$ and $d_2=(m,t_2')$. So $\tilde{M}^m_{1,1}$ has both $t_1'^{-1}$ term and $t_2'^{-1}$ term and $t_1' \ne t_2'$.

Let $d$ be the largest integer such that $(t_1'-t_2')/d$ is an integer class. Since $\tilde{M}$ is Perron-Frobenius, there exists prime number $p$ such that $\tilde{M}^p_{1,1}$ has two terms $t_1^{-1}$ and $t_2^{-1}$ such that $t_1-t_2=t_1'-t_2'$, and $p>max\{k+m,n_0,d,|Tor(H_1(N;\mathbb{Z}))|\}$. This gives closed orbits $c_1=(p,t_1)$ and $c_2=(p,t_2)$. Since $p>d$ and $p$ is prime, there exists positive integer $n$ such that $p|nd+1$.

Proof of (a): For any positive integer $k$, $\tilde{M}^{(kpd+nd+1)p}_{1,1}$ gives possibly nonprimitive closed orbits $((kpd+nd+1)p,(kpd+nd+1)t_1+x(t_2-t_1))$, here $x$ is chosen from $\{0,\cdots,kpd+nd+1\}$. Now we choose $x$ such that $gcd(x,kpd+nd+1)=1$. Since $p|kpd+nd+1$, $gcd(x,(kpd+nd+1)p)=1$. For any prime factor of $(kpd+nd+1)p$, it is factor of every coordinate of $(kpd+nd+1)t_1$, but is not a factor of some coordinated of $x(t_2-t_1)$ since $gcd(d,(kpd+nd+1))=1$. So $N$ has primitive closed orbits with homology $c=((kpd+nd+1)p,(kpd+nd+1)t_1+x(t_2-t_1))$ for any positive integer $k$ and $x\in \{0,\cdots,kpd+nd+1\}$ with $gcd(x,kpd+nd+1)=1$.

Now we need only to show there are infinitely many different drilling classes herein. For $c+x=((kpd+nd+1)p+n_0,(kpd+nd+1)t_1+x(t_2-t_1)+t_0)$, the first coordinate can be rewritten as $kp^2d+(npd+p+n_0)$. Let $d'=gcd(p^2d,npd+p+n_0)$, since $p>n_0$, we have $d'\leq d<p$. By Dirichlet Theorem on arithmetic progressions, there are infinitely many positive integers $k_i$ such that $(k_ip^2d+(npd+p+n_0))/d'$ is prime number. Suppose there are only finitely many pairwise independent classes for all choice of $((k_ipd+nd+1)p+n_0,(k_ipd+nd+1)t_1+x(t_2-t_1)+t_0)$ with $x\in \{0,\cdots,k_ipd+nd+1\}$ and $gcd(x,k_ipd+nd+1)=1$. Then for $i$ large enough, prime number $(k_ip^2d+(npd+p+n_0))/d'$ must be a factor of all coordinates of $((k_ipd+nd+1)p+n_0,(k_ipd+nd+1)t_1+x(t_2-t_1)+t_0)$.

Since we have $\phi(k_ipd+nd+1)$ choices of $x$, here $\phi(\cdot)$ is Euler's totient function, there exists $x_1$ and $x_2$ coprime with $k_ipd+nd+1$ and $0<x_1-x_2\leq \frac{k_ipd+nd+1}{\phi(k_ipd+nd+1)-1}$. Since $(k_ip^2d+(npd+p+n_0))/d'$ is a factor of all coordinates of both $((k_ipd+nd+1)p+n_0,(k_ipd+nd+1)t_1+x_1(t_2-t_1)+t_0)$ and $((k_ipd+nd+1)p+n_0,(k_ipd+nd+1)t_1+x_2(t_2-t_1)+t_0)$, it is a factor of all coordinate of $(x_1-x_2)(t_1-t_2)$. Let $D$ be the upper bound of the norm of all coordinates of $t_1-t_2=t_1'-t_2'$, then we have $k_ipd+nd+1<\frac{k_ip^2d+npd+p+n_0}{d'}\leq D(x_1-x_2)\leq D \frac{k_ipd+nd+1}{\phi(k_ipd+nd+1)-1}$, so $\phi(k_ipd+nd+1)-1\leq D$. Since $\{k_ipd+nd+1\}$ is an integer sequence going to infinity, this is absurd. So we have infinitely many different drilling equivalent classes.

Proof of (b): Choose a prime number $q$ such that $q$ is coprime with both $|Tor(H_1(N;\mathbb{Z}))|$ and $p$. For any positive integer $k$, $\tilde{M}^{(kpd+nd+1)pq}_{1,1}$ gives possibly nonprimitive closed orbit $c$ with homology $((kpd+nd+1)pq,(kpd+nd+1)qt_1+xq(t_2-t_1))$. Here $x$ is chosen from $\{0,\cdots,kpd+nd+1\}$ and coprime with $xpd+nd+1$. As the proof of (a), we can show $d(c)=q$, which is coprime with $|Tor(H_1(N;\mathbb{Z}))|$.

Furthermore, we can take the closed orbit $c$ to go along $c_2$ for $xq$ times first, and then go along $c_1$ for $(kpd+nd+1-x)q$ times. Since $d(c)=q$ is a prime number, $c$ is a nonprimitive closed orbit implies that $c$ repeats $q$ times of a primitive closed orbit. However, it contradicts with the choice of $c$ and the fact that $x\ne 0$ and $x\ne kpd+nd+1$. So we can choose $c$ to be a primitive closed orbit, and it satisfies the condition in Branched Covering Theorem.

Now $d(c)x+(d(c)-1)c=((kpd+nd+1)pq(q-1)+qn_0,(kpd+nd+1)q(q-1)t_1+xq(q-1)(t_2-t_1)+qt_0)$. The first term is $q[(p^2d(q-1))k+((nd+1)p(q-1)+n_0)]$. Let $d'=gcd(p^2d(q-1),(nd+1)p(q-1)+n_0)$ then $d'\leq d(q-1)<p(q-1)$. By Dirichlet Theorem again, there are infinitely many $k_i$ such that $\frac{(p^2d(q-1))k_i+((nd+1)p(q-1)+n_0)}{d'}$ is prime number. Then the following proof is same with the proof of (a).
\end{proof}

Although we do not have an explicit example in hand yet, Theorem \ref{drill} , Theorem \ref{branch} and Lemma \ref{infty} imply that there exists pseudo-Anosov map $(S,\phi)$ (on either closed surface or surface with boundary) such that $A(S,\phi)\ne \mathbb{Q}$, which answers McMullen's question. Moreover, it also shows that pseudo-Anosov map with irrational invariant appears as a general phenomenon.

\section{An Explicit Example}\label{example}

In this section, we will give some explicit examples $(S,\phi)$ with $A(S,\phi)\ne \mathbb{Q}$.

\subsection{An Example and its Closed Orbits}
We will study an explicit example in this subsection. For this example, all the possible homology classes can be realized by a primitive closed orbit. In the next subsection, we will study some simple drilling classes and branched covering classes for this example, and apply an alternative method of showing irrationality.

Let $T_c$ be the left hand Dehn-twist along simple closed curve $c\subset S$. Using Dehn-twists, Penner gave a construction of pseudo-Anosov maps in \cite{Pe}.

\begin{thm}\label{penner}(\cite{Pe} Theorem 3.1)
Suppose $\{a_i\}_{i=1}^m$ and $\{b_j\}_{j=1}^n$ be two families of disjoint essential simple closed curves on surface $S$, such that $\{a_i\}_{i=1}^m$ intersect $\{b_j\}_{j=1}^n$ essentially and every component of $S\setminus ((\cup a_i)\cup(\cup b_j))$ is a disk or anuulus containing a boundary component of $S$. Let $\mathcal{D}(a^+,b^-)$ be the semigroup generated by $T_{a_i}$ and $T_{b_j}^{-1}$. If every $T_{a_i}$ and $T_{b_j}^{-1}$ appear in the presentation of some $\phi \in \mathcal{D}(a^+,b^-)$, then $\phi$ is a pseudo-Anosov map, and an invariant bigon track can be constructed explicitly.
\end{thm}

Let $\phi=T_{a_3}\cdot T_{b_2}^{-1}\cdot T_{b_1}^{-1}\cdot T_{a_2}\cdot T_{a_1}$ be a surface self-homeomorphism on $S$ with $a_i$, $b_j$ as shown in Figure 2. By Theorem \ref{penner}, $T$ is a pseudo-Anosov map. Since $S$ is a closed genus $2$ surface, $A(S,\phi)=\mathbb{Q}$ by Corollary \ref{special}. Although $(S,\phi)$ itself is not an object we are interested in, we will apply our drilling and branched covering construction to it.

\begin{center}
\psfrag{a}[]{\color{blue} $a_1$} \psfrag{b}[]{\color{blue} $a_2$} \psfrag{c}[]{\color{blue} $a_3$}
\psfrag{d}[]{\color{red} $b_1$} \psfrag{e}[]{\color{red} $b_2$} \psfrag{S}[]{$S$} \psfrag{f}[]{$\gamma$}
\includegraphics[width=4in]{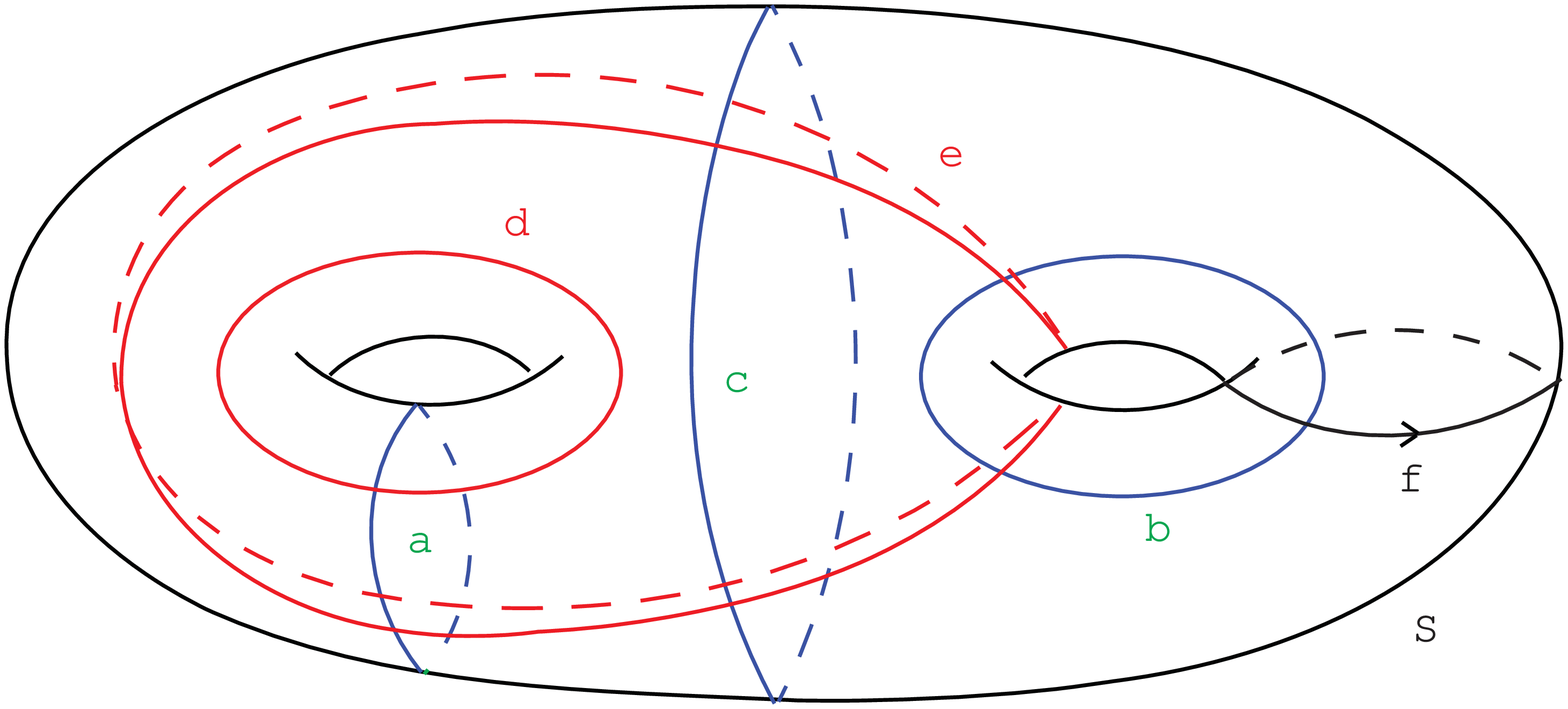}
\vskip 0.5 truecm
 \centerline{Figure 2}
\end{center}

Let $N=M(S,\phi)$ be the mapping torus with $p:N\rightarrow S^1$. It is easy to check that
$H_1(N;\mathbb{Z})=\mathbb{Z}^2=\mathbb{Z}[u] \oplus \mathbb{Z}[t]$, here $p_*(u)$ generates $H_1(S^1;\mathbb{Z})$ and $u\cap [S]=1$, while $t$ is presented by curve $\gamma$ in Figure 2.

We can cut the surface $S$ along curves $a,b,c,d$ in Figure 3 (a) to get an octagon representation of $S$ as in Figure 3 (b). The twisting curves $a_i$, $b_j$ are also shown in Figure 3 (b). The oriented curve $d$ is homologous to $\gamma$ in Figure 2, thus presents the homology class $t$.

\begin{center}
\psfrag{a}[]{$a$} \psfrag{b}[]{$b$} \psfrag{c}[]{$c$} \psfrag{d}[]{$d$} \psfrag{e}[]{$a$} \psfrag{f}[]{$a$}\psfrag{g}[]{$b$} \psfrag{h}[]{$b$} \psfrag{i}[]{$c$} \psfrag{j}[]{$c$} \psfrag{k}[]{$d$} \psfrag{l}[]{$d$} \psfrag{m}[]{(a)} \psfrag{n}[]{(b)} \psfrag{o}[]{\color{blue} $a_1$} \psfrag{p}[]{\color{blue} $a_2$} \psfrag{q}[]{\color{blue} $a_3$} \psfrag{r}[]{\color{blue} $a_3$} \psfrag{s}[]{\color{blue} $a_3$} \psfrag{t}[]{\color{blue} $a_3$} \psfrag{u}[]{\color{red} $b_1$}
\psfrag{v}[]{\color{red} $b_2$} \psfrag{w}[]{\color{red} $b_2$} \psfrag{x}[]{\color{red} $b_2$}
\psfrag{y}[]{\color{red} $b_2$} \psfrag{z}[]{\color{red} $b_2$} \psfrag{aa}[]{\color{red} $b_2$}
\includegraphics[width=5in]{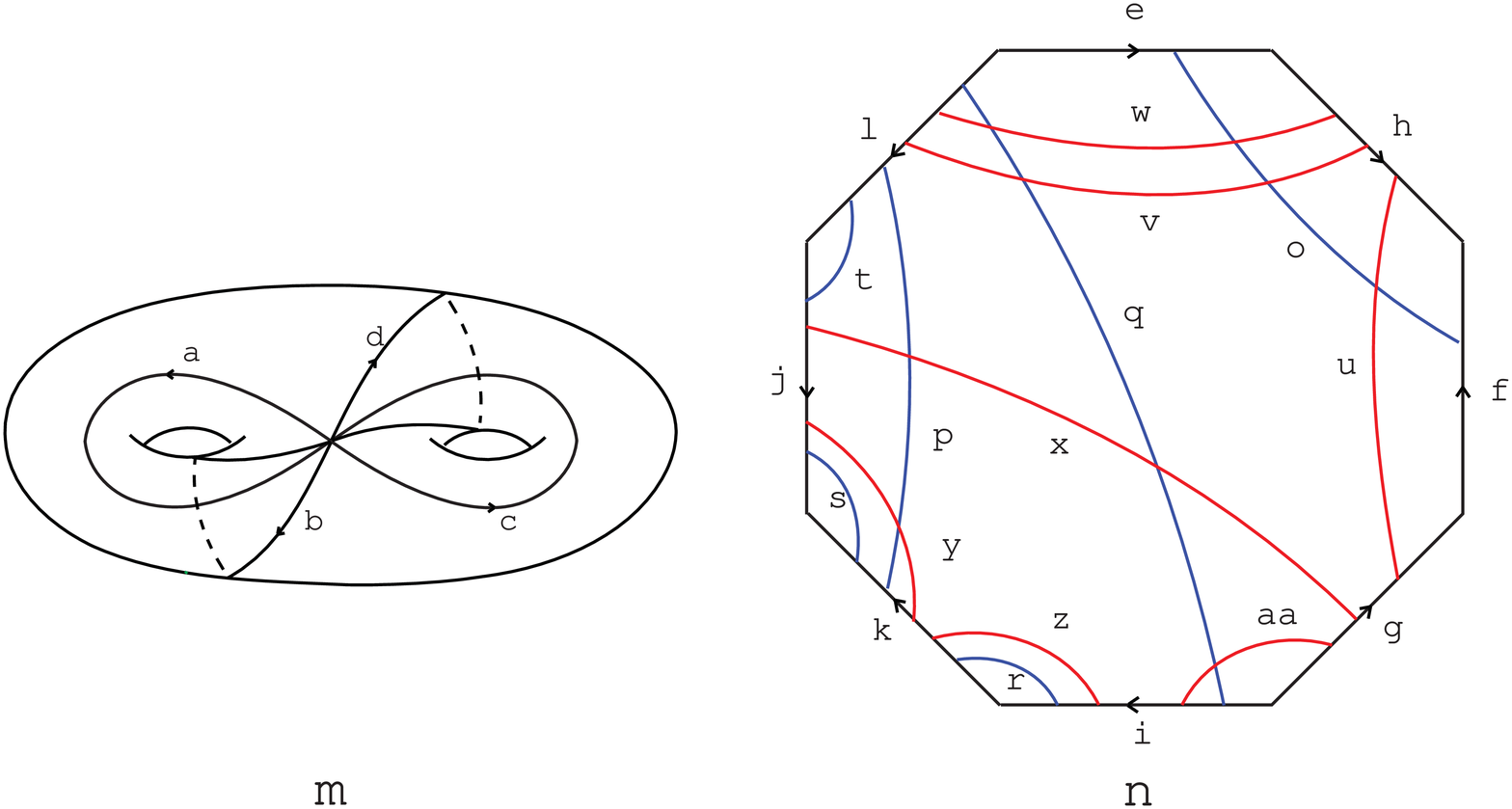}
\vskip 0.5 truecm
 \centerline{Figure 3}
\end{center}

By running Bestvina-Handel's algorithm (\cite{BH}), we get a graph $G\subset S$ in Figure 4 (a), which is the spine of a {\it fibered surface} carrying $\phi$, such that the induced map $\hat{\phi}:G\rightarrow G$ is {\it efficient} (see \cite{BH} for terminology show up here). We also get an invariant train track $\tau$ as shown in Figure 4 (b).

\begin{center}
\psfrag{a}[]{$c$} \psfrag{b}[]{$d$} \psfrag{v}[]{$v$}
\psfrag{w}[]{$w$} \psfrag{x}[]{$x$} \psfrag{y}[]{$y$} \psfrag{z}[]{$z$}
\psfrag{c}[]{(a)} \psfrag{d}[]{(b)}
\includegraphics[width=5in]{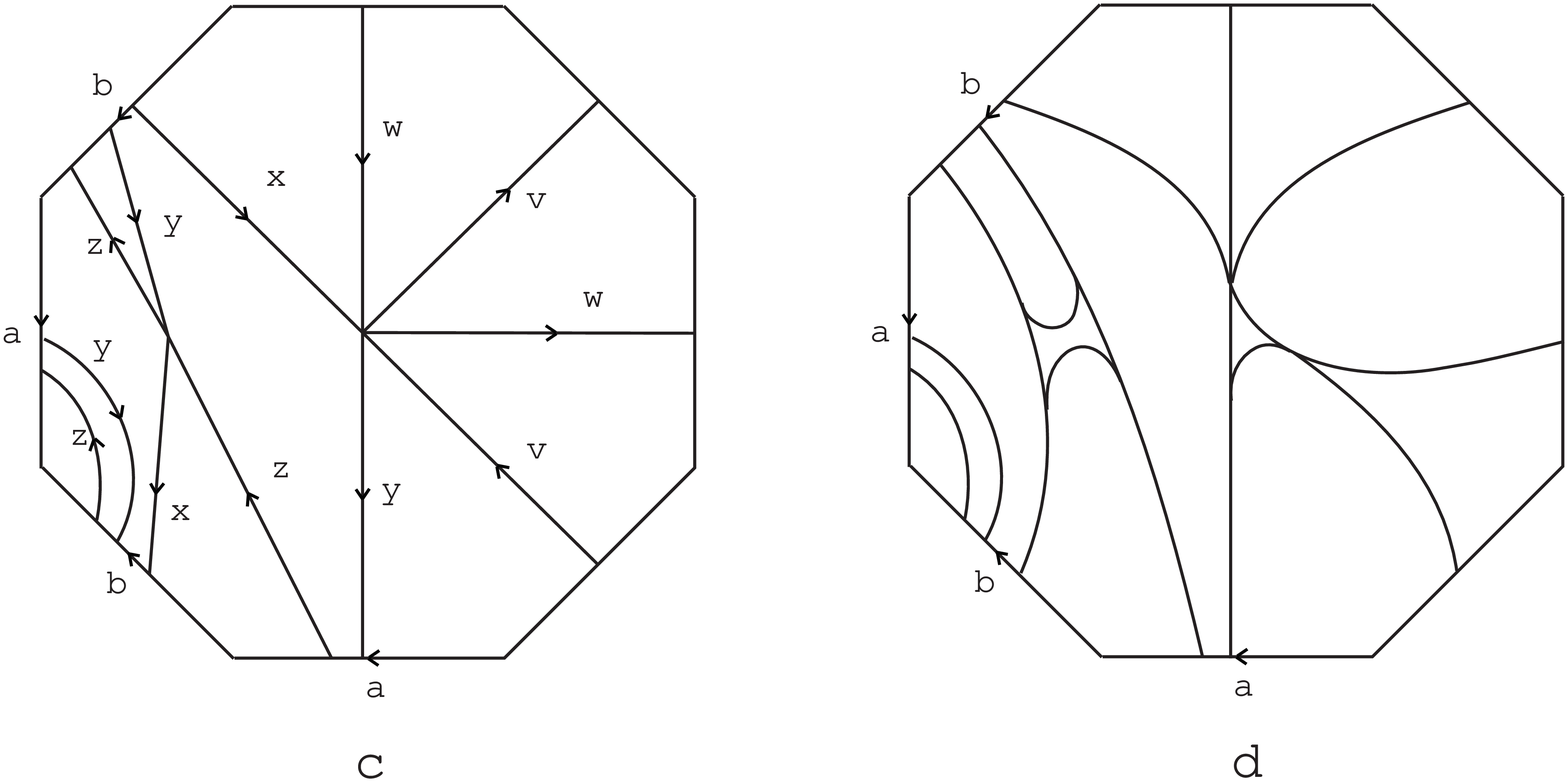}
\vskip 0.5 truecm
 \centerline{Figure 4}
\end{center}

Moreover, we can thicken each edge of $G$ to get a Markov partition of $\phi$ as in \cite{BH}.

Let $\bar{a}$ be the inverse of oriented edge $a$, then the induced map on graph $\hat{\phi}:G\rightarrow G$ is as the following:
$$
\left\{ \begin{array}{l}
         v \rightarrow vy\bar{z}\bar{y}\bar{x}zxyzxyz\bar{y}\bar{x}\bar{z}xyz\bar{y}\bar{v} y\bar{z}\bar{y}\bar{x}\bar{z}\bar{y}\bar{x}\bar{z}xyz\bar{y}v y\bar{z}\bar{y}\bar{x}zxy\bar{z}\bar{y}\bar{x}\bar{z}\bar{y}\bar{x}\bar{z}xyz\bar{y}\bar{v} y\bar{z}\bar{y}\bar{x}zxyzxyz\bar{y}\bar{w}v^2, \\
         w \rightarrow \bar{v}w y\bar{z}\bar{y}\bar{x}\bar{z}\bar{y}\bar{x}\bar{z}xyz\bar{y}v y\bar{z}\bar{y}\bar{x}zxyzxyz\bar{y}\bar{x}\bar{z}xyz\bar{y}\bar{v} y\bar{z}\bar{y}\bar{x}zxyzxyz\bar{y}v y\bar{z}\bar{y}\bar{x}zxy\bar{z}\bar{y}\bar{x}\bar{z}\bar{y}\bar{x}\bar{z}xyz\bar{y}\bar{v}, \\
         x \rightarrow xyz\bar{y}v y\bar{z}\bar{y}\bar{x}zxy\bar{z}\bar{y}\bar{x}\bar{z}\bar{y}\bar{x}\bar{z}xyz\bar{y}\bar{v}, \\
         y \rightarrow y\bar{z}\bar{y}\bar{x}zxyzxyz\bar{y}\bar{x}\bar{z}xyz\bar{y}\bar{v} y\bar{z}\bar{y}\bar{x}\bar{z}\bar{y}\bar{x}\bar{z}xyz\bar{y}v y\bar{z}\bar{y}\bar{x}zxy, \\
         z \rightarrow zxyz.
\end{array} \right.
$$

Let $\tilde{S}$ be one component of the lift of $S$ in the maximal abelian cover $\tilde{M}$, and $\tilde{\phi}$ be the lift of $\phi$. Since $H_1(N;\mathbb{Z})=\mathbb{Z}[u]\oplus \mathbb{Z}[t]$ (take $t=[d]$), $\tilde{S}$ is obtained by cutting $S$ along $c$, then paste $\mathbb{Z}$-copies together along $c$. Let $\tilde{G}$ be the preimage of $G$ in $\tilde{S}$, then an abstract picture of $\tilde{G}$ is shown in Figure 5. 

\begin{center}
\psfrag{a}[]{$t^{-1}z$} \psfrag{b}[]{$t^{-1}y$} \psfrag{c}[]{$x$} \psfrag{d}[]{$t^{-1}v$} \psfrag{e}[]{$t^{-1}w$}
\psfrag{f}[]{$z$} \psfrag{g}[]{$y$} \psfrag{h}[]{$tx$} \psfrag{i}[]{$v$} \psfrag{j}[]{$w$}
\psfrag{k}[]{$tz$} \psfrag{l}[]{$ty$} \psfrag{m}[]{$t^2x$} \psfrag{n}[]{$tv$} \psfrag{o}[]{$tw$}
\includegraphics[width=5.5in]{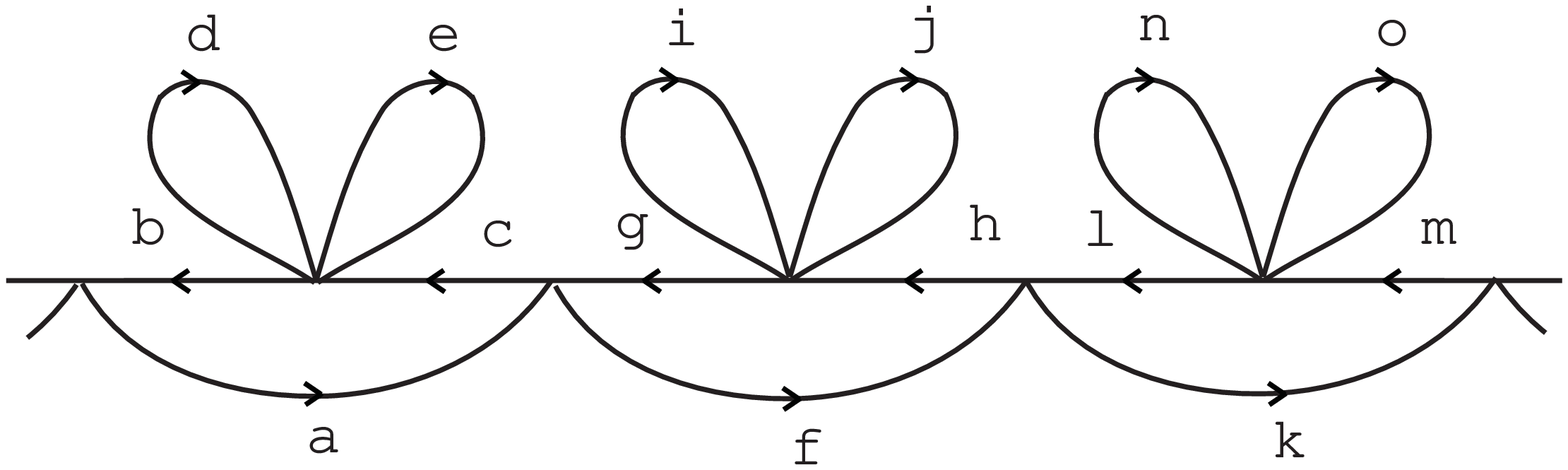}
\vskip 0.5 truecm
 \centerline{Figure 5}
\end{center}

Since we do not need to keep track the path of $\hat{\tilde{\phi}}:\tilde{G}\rightarrow \tilde{G}$ in the following work, we will write the formula of $\hat{\tilde{\phi}}$ in addition form, but not as composition of paths. We will also omit the orientation of paths.
Using the formula of $\hat{\phi}$ and Figure 5, we get the following transition matrix for $\hat{\tilde{\phi}}:\tilde{G}\rightarrow \tilde{G}$. Let $x_1,\cdots,x_5$ be $v,w,x,y,z$ respectively, then the entry $m_{i,j}$ is the sum of edges in $\mathbb{Z}[t](\hat{\tilde{\phi}}(x_j))$ collapsing to $x_i$.

$$M(t)=
\left( \begin{array}{ccccc}
t+4+t^{-1} &  t+3+t^{-1}&  t+1 & 1+t^{-1} & 0 \\
1 & 1 & 0 & 0& 0\\
2t+7+6t^{-1}+t^{-2} & 2t+7+6t^{-1}+t^{-2} & t+4+t^{-1} & 3+6t^{-1}+t^{-2} & t^{-1} \\
2t^2+9t+10+3t^{-1} & 2t^2+9t+10+3t^{-1} & t^2+5t+3 & 3t+9+3t^{-1} & 1\\
2t^2+9t+8+t^{-1} & 2t^2+9t+8+t^{-1} & t^2+5t+1 & 3t+8+t^{-1} & 2
 \end{array} \right).
$$

In Figure 4 (b), for the invariant train track $\tau$, there are seven switches and seven branches other than $v,w,x,y,z$, and $\phi$ (also $\tilde{\phi}$) fixes all of them. So the contribution of seven switches to $det(uI-P_V(t))$ cancels the contribution of seven branches to $det(uI-P_E(t))$. Then we have $\Theta_F(u,t)=\frac{det(uI-P_E(t))}{det(uI-P_V(t))}=det(uI-M(t))$. By computing the characteristic polynomial of $M(t)$, up to a unit in $\mathbb{Z}[H_1(N;\mathbb{Z})/Tor]$, we get:
\begin{equation}\label{Teichmuller}
\Theta_F(u,t) =(u-1)(u^2-(5t+19+5t^{-1})u+(14t+48+14t^{-1})-(5t+19+5t^{-1})u^{-1}+u^{-2}).
\end{equation}

Let $(\alpha_1,\alpha_2)$ be a basis of $H^1(N,\mathbb{R})$ dual with $(u,t)$, we have $\alpha_1=[S]$. Then by Theorem \ref{cone} and formula (\ref{Teichmuller}), we have that $C=\mathbb{R}_+\cdot F=\mathbb{R}_+\cdot D=\{x_1\alpha_1+x_2\alpha_2|\  x_1>|x_2|\}.$ For any closed orbit $c$ of the suspension flow of $N=M(S,\phi)$, we have that for any $\alpha \in C$, $\langle \alpha, c \rangle > 0$. So $c=au+bt$ with $a\in \mathbb{Z}_+, b\in \mathbb{Z}$ and $a\geq |b|$.

\begin{prop}
For any homology class $au+bt$ with $a\in \mathbb{Z}_+, b\in \mathbb{Z}$ and $a\geq |b|$, there exists a primitive closed orbit $c$ in $N=M(S,\phi)$, such that $[c]=au+bv$.
\end{prop}

\begin{proof}
The transition matrix $M(t)$ also gives the transition matrix of $\tilde{\phi}$ under the Markov partition $\cup_{i=1}^5 \mathbb{Z}[t](\tilde{R_i})$ given by $\tilde{G}$. Since $m_{4,4}=3t+9+3t^{-1}$, we know that $\tilde{\phi}^{-1}(\tilde{R_4})\cap \tilde{R_4}$ has nine components, while $\tilde{\phi}^{-1}(\tilde{R_4})\cap t\cdot \tilde{R_4}$ and $\tilde{\phi}^{-1}(\tilde{R_4})\cap t^{-1}\cdot\tilde{R_4}$ both have three components.

For any homology class $au+bt$ with $b\ne 0$ and $a\ne |b|$, let's suppose $b>0$, then we can take $x \in R_4$, with lifting $\tilde{x}\in \tilde{R_4}$, such that:\\
1) $\tilde{\phi}^i(\tilde{x})\in t^i\cdot \tilde{R_4}$ for $0\leq i \leq b$, \\
2) $\tilde{\phi}^j(\tilde{x})\in t^b\cdot \tilde{R_4}$ for $b< j \leq a$,\\
3) $\tilde{\phi}^a(\tilde{x})= t^b\cdot \tilde{x}.$\\
By 3), $x$ is a period point with primitive closed orbit $c$ passing through it with $kc=au+bt$ for some $k\in \mathbb{Z}_+$. Actually, the period of $x$ is $a$. Otherwise, by 1) and 2), either $\tilde{\phi}^a(\tilde{x})= \tilde{x}$ or $\tilde{\phi}^a(\tilde{x})= t^a\cdot \tilde{x}$, contradicts with $b\ne 0$ and $a\ne |b|$. So $k=1$.

Then we need only to deal with homology classes $au$, $au+at$ and $au-at$. Since the proof of these three cases are similar, we will only prove the $au+at$ case.
Since $\tilde{\phi}^{-1}(t\cdot \tilde{R_4})\cap \tilde{R_4}$ has three components, we denote them by $S_1,S_2,S_3$. Then there exists $\tilde{x} \in S_1$, such that:\\
1) $\tilde{\phi}(\tilde{x})\in t \cdot S_2$, \\
2) $\tilde{\phi}^i(\tilde{x})\in t^i\cdot S_1$ for $1< i \leq a$,\\
3) $\tilde{\phi}^a(\tilde{x})=t^a \cdot \tilde{x}.$\\
By the same argument as above, $x$ has period $a$ and the primitive closed orbit passing through $x$ has homology $au+at$.
\end{proof}

Since $\alpha_1$ is presented by fiber surface $S$, we have $\|\alpha_1\|=2$. Since $S$ is the closed surface with genus $2$, by Corollary \ref{special}, $N=M(S,\phi)$ admits an involution $\bar{\tau}$ with $\bar{\tau}_*(u)=u$, $\bar{\tau}_*(t)=-t$. So $\bar{\tau}^*(\alpha_1)=\alpha_1$, $\bar{\tau}^*(\alpha_2)=-\alpha_2$. It implies that $\|\alpha_1+t\alpha_2\|=\|\alpha_1-t\alpha_2\|$ for any $t \in \mathbb{R}$. Since $\alpha_1+t\alpha_2$ lies in the interior of fibered cone $C$ for any $t\in (-1,1)$, $\|\alpha_1+t\alpha_2\|=2$ for $t\in (-1,1)$. Thus the corresponding open fibered face $F=\{\frac{1}{2}\alpha_1+t\alpha_2|\ t\in(-\frac{1}{2},\frac{1}{2})\}$, and the dual of this fibered face is $x=2u\in H_1(N,\mathbb{Z})$.

\subsection{An Alternative Method for Irrationality}\label{numerical}
Although the Drilling Theorem (Theorem \ref{drill}), Branched Covering Theorem (Theorem \ref{branch}) and Lemma \ref{infty} show us there are infinitely many pseudo-Anosov maps (on closed surface or surface with boundary) with $A(S,\phi)\ne \mathbb{Q}$, we do not know for which class $c$, $A(S_c,\phi_c)\ne\mathbb{Q}$ holds. The constants $D$ and $D'$ in the proof of both theorems are very large, and it is difficult to find which drilling classes (branched covering classes) are the exceptional classes.

We will use an alternative method to show that certain drilling (branched covering) classes have irrational invariant. Assuming the minimal point is rational, we use algebraic number theory to bound the denominator of some function of minimal point coordinate, then we compute the function numerically and show that it can not be a rational number with denominator under the given bound. Since we use numerical method here, it only gives a case-by-case argument and provides some special examples, but may not help us to understand the invariant deeper.

We will still work on the example constructed in the previous subsection and we use all the notations therein.

Let's first deal with the drilling case, choose the simplest drilling class $c=u+t$, then the corresponding $F_c=\{\alpha\in C|\ \langle \alpha, c+x \rangle=1\}=\{a\alpha_1+b\alpha_2|\ a>|b|, 3a+b=1\}$. Let the minimal point $m_{F_c}=s\alpha_1+(1-3s)\alpha_2$ ($s\in (1/4,1/2)$), and let $\lambda=\lambda(m_{F_c})$. Plugging in the formula (\ref{Teichmuller}) and omitting the $u-1$ factor, then take derivative as in Proposition \ref{transcendental}, we have the following equation:
{\small \begin{equation}\label{1}
\left\{ \begin{array}{l}
        \lambda^{2s}-5\lambda^{1-2s}-19\lambda^{s}-5\lambda^{4s-1}+14\lambda^{1-3s}+48+14\lambda^{3s-1}-5\lambda^{1-4s}-19\lambda^{-s}-
        5\lambda^{2s-1}+\lambda^{-2s}=0, \\
        2\lambda^{2s}+10\lambda^{1-2s}-19\lambda^{s}-20\lambda^{4s-1}-42\lambda^{1-3s}+42\lambda^{3s-1}+20\lambda^{1-4s}+19\lambda^{-s}-
        10\lambda^{2s-1}-2\lambda^{-2s}=0.
\end{array} \right.
\end{equation}}

Let $\lambda=X, \lambda^s=Y$, then we have:
{\small \begin{equation}\label{2}
\left\{ \begin{array}{l}
        -(5Y^{-2}-14Y^{-3}+5Y^{-4})X+(Y^2-19Y+48-19Y^{-1}+Y^{-2})-(5Y^4-14Y^3+5Y^2)X^{-1}=0, \\
        (10Y^{-2}-42Y^{-3}+20Y^{-4})X+(2Y^2-19Y+19Y^{-1}-2Y^{-2})+(-20Y^4+42Y^3-10Y^2)X^{-1}=0.
\end{array} \right.
\end{equation}}

By solving the first quadratic equation of (\ref{2}), we get: $X=$
{\small \begin{equation}\label{3}
\frac{(Y^2-19Y+48-19Y^{-1}+Y^{-2})\pm \sqrt{(Y^2-9Y+20-9Y^{-1}+Y^{-2})(Y^2-29Y+76-29Y^{-1}+Y^{-2})}}{10Y^{-2}-28Y^{-3}+10Y^{-4}}.
\end{equation}}

By plugging in formula (\ref{3}) into the second equation of (\ref{2}), let $A=Y+Y^{-1}$ and simplify the equation, we get:
\begin{align}\label{4}
0=f(A)=200A^6-9530A^5+128025A^4-778216A^3+2422552A^2-3782016A+2354832 \nonumber \\
=(15A-42)^2(A^2-9A+18)(A^2-29A+74)-(A^2-4)(5A^2-28A+36)^2.
\end{align}

$A$ is a positive real root of equation (\ref{4}), with
\begin{align}\label{5}
Y=\frac{A+\sqrt{A^2-4}}{2},
\end{align} and
\begin{align}\label{6}
XY^{-3}=\frac{(A^2-19A+46)\pm \sqrt{(A^2-9A+18)(A^2-29A+74)}}{10A-28} \nonumber \\
=\frac{3(5A-14)(A^2-19A+46)\pm (5A^2-28A+36)\sqrt{A^2-4}}{6(5A-14)^2}.
\end{align}

So $X,Y$ lie in number field $\mathbb{F}$ with $[\mathbb{F}:\mathbb{Q}(A)]\leq 2$. Since $A$ is the root of degree $6$ polynomial $f(x)$ in equation (\ref{4}), and we can check that $f(x)$ is irreducible modulus $7$ by Mathematica (\cite{Math}), we have $[\mathbb{Q}(A):\mathbb{Q}]=6$. So $[\mathbb{F}:\mathbb{Q}]=6$ or $12$.

Since $A=Y+Y^{-1}$, the minimal polynomial $p(x)$ of $Y$ is a factor of $x^6\cdot f(x+x^{-1})$, so either $p(x)=\pm x^{\deg(p)}p(x^{-1})$, or $p(x)p(x^{-1})|\ x^6\cdot f(x+x^{-1})$. Since $f(x)$ is irreducible, we have either $p(x)=x^6\cdot f(x+x^{-1})$, or $p(x)p(x^{-1})= x^6\cdot f(x+x^{-1})$. Both of these two cases imply that $Y$ is not an algebraic unit since the first or last term of $f(x)$ has coefficient greater than $1$. Moreover, by equation (\ref{4}) and $A=Y+Y^{-1}$, we know that both $200Y$ and $200Y^{-1}$ are algebraic integers.

The definition in algebraic number theory appears in the following paragraphs can be found in \cite{MR} Chapter 0.

Let $\mathcal{O}_{\mathbb{F}}$ be the ring of algebraic integers of number field $\mathbb{F}$, then $Y\mathcal{O}_{\mathbb{F}}\subset \mathbb{F}$ is a {\it fractional ideal} of Dedekind domain $\mathcal{O}_{\mathbb{F}}$. By \cite{MR} Theorem 0.3.4, $Y\mathcal{O}_{\mathbb{F}}$ can be decomposed as production of prime ideals of $\mathcal{O}_{\mathbb{F}}$ and their inversions. Let the decomposition be $Y\mathcal{O}_{\mathbb{F}}=\mathcal{P}_1^{p_1}\cdots\mathcal{P}_m^{p_m}\cdot \mathcal{Q}_1^{-q_1}\cdots\mathcal{Q}_n^{-q_n}$, here $p_i,q_j\in \mathbb{Z}_+$, $\mathcal{P}_i$ and $\mathcal{Q}_j$ are prime ideals of $\mathcal{O}_{\mathbb{F}}$. This decomposition is nontrivial since $Y$ is not an algebraic unit.

Let $200\mathcal{O}_{\mathbb{F}}=(2\mathcal{O}_{\mathbb{F}})^3(5\mathcal{O}_{\mathbb{F}})^2=(\mathcal{P}^2_1)^{3a_1}\cdots(\mathcal{P}^2_m)^{3a_m}\cdot (\mathcal{P}^5_1)^{2b_1}\cdots(\mathcal{P}^5_n)^{2b_n}$. Since $[\mathbb{F}:\mathbb{Q}]\leq 12$, we have $a_i,b_j\leq 12$. Since $200Y \in \mathcal{O}_{\mathbb{F}}$, we have that $\mathcal{Q}_j \in \{\mathcal{P}^2_1,\cdots,\mathcal{P}^2_m, \mathcal{P}^5_1,\cdots,\mathcal{P}^5_n\}$ and $0<q_j\leq 36$. Since $200Y^{-1} \in \mathcal{O}_{\mathbb{F}}$, we also have $0<p_i \leq 36$.

If $su+(1-s)t$ is a rational class with $s=\frac{q}{p}$ and $gcd(p,q)=1$, then $X=Y^{\frac{1}{s}}=Y^{\frac{p}{q}}$. So $X\mathcal{O}_{\mathbb{F}}=\mathcal{P}_1^{\frac{pp_1}{q}}\cdots\mathcal{P}_m^{\frac{pp_m}{q}}\cdot \mathcal{Q}_1^{-\frac{pq_1}{q}}\cdots\mathcal{Q}_n^{-\frac{pq_n}{q}}$, which implies $q\leq 36$. To show that $A(S_c,\phi_c)\ne \mathbb{Q}$, we need only to compute $\frac{\log{X}}{\log{Y}}$ numerically, and check it is not a rational number with denominator less or equal to $36$.

We can solve equation (\ref{4}) by Mathematica (\cite{Math}), and get $A=30.38934206615629\dots$. By plugging in the value of $A$ into equation (\ref{5}) and (\ref{6}), we get $Y=30.35640008366680\dots$, and $X=11506.21849\dots$. So $\frac{\log{X}}{\log{Y}}=2.739707\dots$. It is easy to check that $\frac{\log{X}}{\log{Y}}$ can't be a rational number with denominator $\leq 36$. So for $c=u+t$, $A(S_c,\phi_c)\ne \mathbb{Q}$. Here the minimal point of $\lambda(\cdot)$ on $F_c=\{\alpha \in C|\ \langle \alpha,c+x \rangle=1\}$ is $su+(1-3s)t$ with $s=\frac{\log{Y}}{\log{X}}=0.365002\dots$.

Take primitive class $\beta_n=n\alpha_1-(n-1)\alpha_2\in C$ with $n\in \mathbb{Z}_+$. Since $\langle \beta_n, u+t \rangle=1$, the corresponding fiber surface $S_n$ in $N_c=M(S_c,\phi_c)$ has one boundary component. On the other hand $-\chi(S_n)=\langle \beta_n, x+c \rangle=\langle n\alpha_1-(n-1)\alpha_2, 3u+t \rangle=2n+1$, so $S_n$ is a genus $n+1$ surface with one boundary component. Now we get the following theorem.

\begin{thm}\label{onecusp}
For any genus $g$ surface $S$ with one boundary component and $g\geq 2$, there exists pseudo-Anosov map $\phi$ on $S$, such that $A(S,\phi)\ne\mathbb{Q}$.
\end{thm}

Now let's turn to branched covering case. Since the branched covering class has nonprimitive homology, the simplest choice is $c'=2u+2t$, and $d(c')x+(d(c')-1)c'=6u+2t$. Since $d(c')x+(d(c')-1)c'=6u+2t$ is linear dependent with $c+x=3u+t$ in the previous example, we know that $F^{c'}=\frac{1}{2}F_c$, thus $m_{F^{c'}}=\frac{1}{2}m_{F_c}$. So $A(S^{c'},\phi^{c'})=A(S_c,\phi_c)\ne \mathbb{Q}$.

Let $\beta_n=n\alpha_1-(n-1)\alpha_2$ with $n\in \mathbb{Z}_+$, the corresponding surface $S_n'$ in $N^{c'}=M(S^{c'},\phi^{c'})$ satisfies $-\chi(S_n')=\langle \beta_n, d(c')x+(d(c')-1)c'\rangle=\langle n\alpha_1-(n-1)\alpha_2, 6u+2t \rangle=4n+2$. So $S_n'$ has genus $2n+2$ with $A(S_n',\phi_n')\ne\mathbb{Q}$.

\begin{thm}\label{even}
For any closed genus $2n$ surface $S$ with $n\geq 2$, there exists pseudo-Anosov map $\phi$ on $S$, such that $A(S,\phi)\ne\mathbb{Q}$.
\end{thm}

\begin{rem}
Since there does not exist a genuine branched cover from $\Sigma_{3,0}$ to $\Sigma_{2,0}$, we can not produce a pseudo-Anosov map $\phi$ on $\Sigma_{3,0}$ with $A(\Sigma_{3,0},\phi)\ne \mathbb{Q}$ from the example in this section. We will construct another example in the next section and show a similar theorem of Theorem \ref{even} for odd genus surfaces.
\end{rem}

\section{Penner's Construction}\label{cons2}
In this section, we will study pseudo-Anosov maps given by Penner's construction (\cite{Pe}) as in Theorem \ref{penner}. We will use notations in Theorem \ref{penner}.

Suppose $\phi\in \mathcal{D}(a^+,b^-)$, such that all the $T_{a_i}$ and $T^{-1}_{b_j}$ appear in the presentation of $\phi$, then $\phi$ is pseudo-Anosov. Since the induced map of $T_c$ on homology is given by $(T_c)_*(x)=x+(x\cdot c)c$, we have $\phi_*(x)-x\in span(a_1,\cdots,a_m,b_1,\cdots,b_n)$. So $i_*(H_1(S,\mathbb{R}))\subset H_1(M(S,\phi);\mathbb{R})$ has a natural quotient $H_1(S,\mathbb{R})/span(a_1,\cdots,a_m,b_1,\cdots,b_n)$.

\begin{defn}
Let $\phi$ be a pseudo-Anosov element in $\mathcal{D}(a^+,b^-)$. $\phi$ is said to be {\it generic} if  $i_*(H_1(S,\mathbb{R}))=H_1(S,\mathbb{R})/span(a_1,\cdots,a_m,b_1,\cdots,b_n)$, i.e. the quotient in the previous paragraph is trivial.
\end{defn}

For a generic $\phi$, we can read the homology of $H_1(M(S,\phi);\mathbb{R})$ directly from twisting curves $a_i,b_j$, but do not need to know more information about $\phi$.

\subsection{Polynomial $\Phi$ and Dilatation Function $\lambda(\cdot)$}
In this subsection, for a generic element $\phi\in \mathcal{D}(a^+,b^-)$, we will define another polynomial $\Phi$ which can also compute the dilatation function $\lambda(\cdot)$ effectively. Moreover, $\Phi$ can be computed directly from the two families of curves $\{a_i\}_{i=1}^m,\{b_j\}_{j=1}^n$ and the presentation of $\phi$, but we do not need to construct the invariant train track of $\phi$.

In \cite{Pe} Theorem 3.1, Penner constructed an invariant bigon-track $\tau$ of pseudo-Anosov map $\phi$. Actually, $\tau$ only depends on the two families of curves $\{a_i\}_{i=1}^m$, $\{b_j\}_{j=1}^n$, but does not depend on the specific element $\phi\in \mathcal{D}(a^+,b^-)$.

For each bigon region $B_k$ in $S\setminus \tau$, we choose a point $p_k\in B_k$, and let $S'=S\setminus \{p_1,\cdots,p_k\}$ with inclusion $i:S'\rightarrow S$. Let $\{a_i'\}_{i=1}^m=\{i^{-1}(a_i)\}_{i=1}^m,\{b_j'\}_{j=1}^n=\{i^{-1}(b_j)\}_{j=1}^n$ be the corrsponding two families of disjoint simple closed curves on $S'$. Then $\{a_i'\}_{i=1}^m, \{b_j'\}_{j=1}^n$ still satisfy the assumption of Theorem \ref{penner}. Let $\phi' \in \mathcal{D}(a'^+,b'^-)$ be the mapping class of $S'$ corresponding with $\phi \in \mathcal{D}(a^+,b^-)$, then we have $i\circ \phi'=\phi\circ i$, and $\phi'$ is also a pseudo-Anosov map.

\begin{lem}\label{nice1}
Let $N=M(S,\phi)$ and $N'=M(S',\phi')$ be the mapping torus with inclusion $i':N'\rightarrow N$. Let $C\subset H^1(N;\mathbb{R})$ be the fibered cone containing the dual of $[S]$. Then for any $\alpha\in C$, we have $\lambda(i^*(\alpha))=\lambda(\alpha)$. Moreover, the homology class dual to fibered cone $F$ is $|\chi(S)|u$ for some $u\in H_1(N;\mathbb{Z})$ and $u\cap [S]=1$.
\end{lem}

\begin{proof}
From $\{a_i'\}_{i=1}^m, \{b_j'\}_{j=1}^n$, by using Penner's method, we can construct an invariant bigon track $\tau'=i^{-1}(\tau)$. $\tau'$ is actually a train track since we have added a puncture on each bigon complement of $S\setminus \tau$. A pair of transverse $\phi'$-invariant singular foliation $(\mathcal{F}'^+,\mathcal{F}'^-)$ on $S'$ can be constructed as in \cite{Pe}. We first construct $\phi'$-invariant horizontal and vertical foliations on $N(\tau')$ (a neighborhood of $\tau'$). Then collapse the disc and annulus regions of $S'\setminus N(\tau')$ to get singular foliations $(\mathcal{F}'^+, \mathcal{F}'^-$) on $S'$. Since the region of $S'\setminus \tau'$ corresponding with bigon region $B_k\subset S$ is $B_k-\{p_k\}$, $(\mathcal{F}'^+,\mathcal{F}'^-)$ near the puncture is as shown Figure 6. To make the picture clearer, we draw a boundary component but not a puncture in the picture.

Since $(\mathcal{F}'^+,\mathcal{F}'^-)$ give a "$2$-prong"-picture near puncture $p_k$ as in Figure 6, $(\mathcal{F}'^+,\mathcal{F}'^-)$ give a pair of transverse $\phi$-invariant singular foliation $(\mathcal{F}^+,\mathcal{F}^-)$ on $S$. $\phi$ preserves the pair of singular foliations $(\mathcal{F}^+,\mathcal{F}^-)$, and each $p_k$ is a fixed point of the pseudo-Anosov map $\phi$ on $S$. Let $s_1,\cdots,s_t$ be singular points of transverse singular foliation $(\mathcal{F}^+,\mathcal{F}^-)$ on $S$. By the construction of $(\mathcal{F}'^+,\mathcal{F}'^-)$, all these $s_i$ are fixed point of $\phi$.

Now we turn to $3$-manifolds $N=M(S,\phi)$ and $N'=M(S',\phi')$. Since the punctures $p_k$ are all fixed point of $\phi$, let $c_k$ be the closed orbit passing through $p_k$, then $N'$ is obtained from $N$ by drilling closed orbits $c_1,\cdots,c_l$ of suspension flow. So $\lambda(i^*(\alpha))=\lambda(\alpha)$ for any $\alpha \in C$.

\begin{center}
\psfrag{a}[]{\color{red} $\mathcal{F}^+$} \psfrag{b}[]{\color{blue} $\mathcal{F}^-$}
\includegraphics[width=3in]{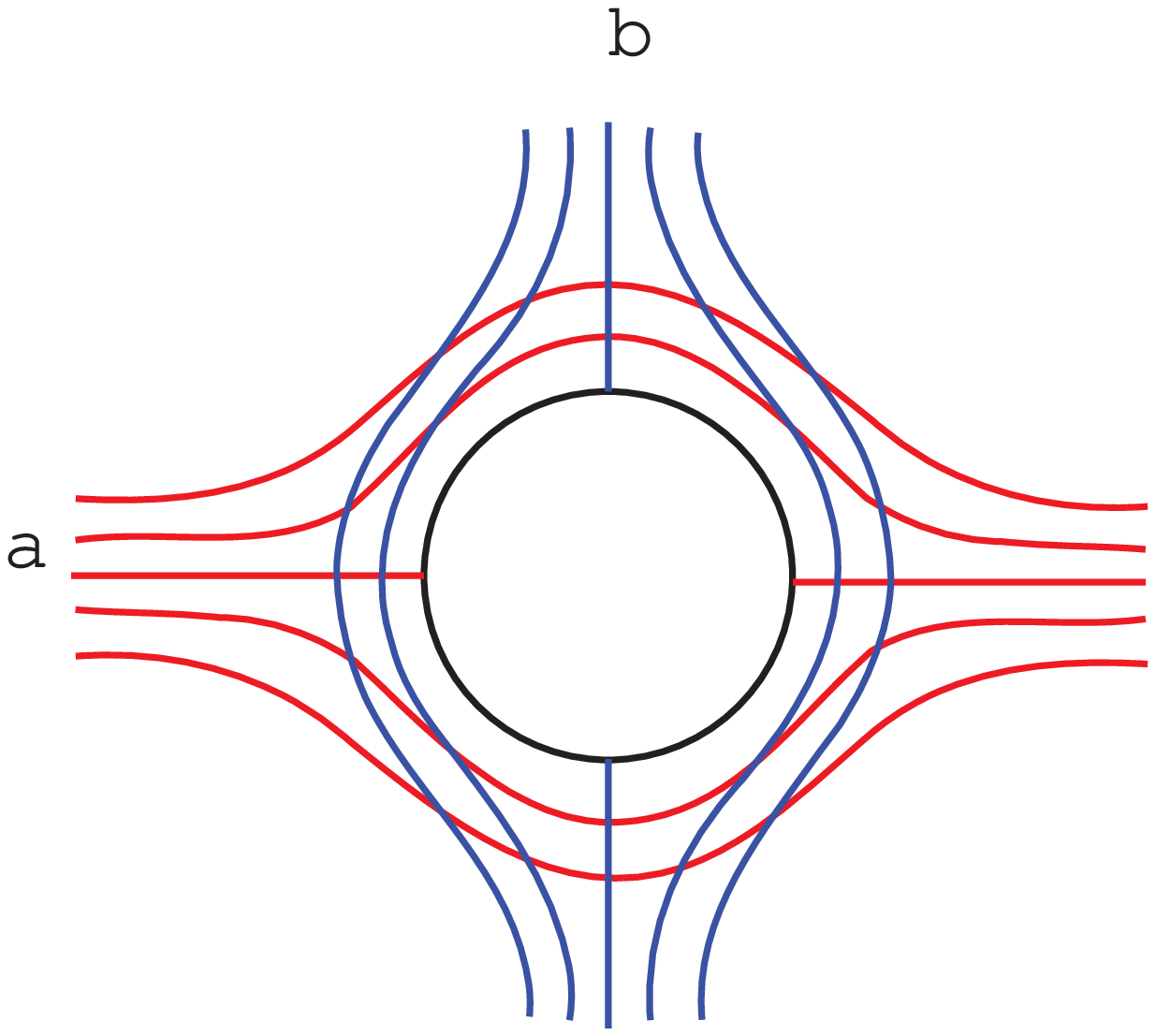}
\vskip 0.5 truecm
 \centerline{Figure 6}
\end{center}

Let $c'_1,\cdots,c'_t$ be the oriented closed orbit passing through $s_1,\cdots,s_t$ respectively with $c'_i\cap S=1$. We will show that all the closed orbits $c'_1,\cdots,c'_t$ share the same homology class in $N$. For any two points in $x,y\in \{s_1,\cdots,s_t\}$, let $\gamma$ be an arc connecting $x$ and $y$, then the corresponding closed orbit $c_x'$ and $c_y'$ satisfy $\gamma\cdot c_y'\cdot \phi(\bar{\gamma})\cdot \bar{c}_x'=0$ in $\pi_1(N)$. So $[c_x']-[c_y']=[\gamma\cdot \phi(\bar{\gamma})]\in H_1(N;\mathbb{Z})$. Since $\phi$ is composition of Dehn twists along $\alpha$ and $\beta$ curves, $[\gamma\cdot \phi(\bar{\gamma})]\in span(a_1,\cdots,a_m,b_1,\cdots,b_n)\subset H_1(N;\mathbb{R})$. Since $\phi$ is generic, $span(a_1,\cdots,a_m,b_1,\cdots,b_n)=\{0\}$ in $H_1(N;\mathbb{R})$, so $[c_x']-[c_y']=[\gamma\cdot \phi(\bar{\gamma})]=0 \in H_1(N;\mathbb{R})$ and $[c_x']=[c_y']$. Take $u\in H_1(N;\mathbb{Z})$, such that $[c'_1]=\cdots=[c'_t]=u$

Any singular point $s_i$ gives a $d_i$-prong singularity of transverse singular foliation $(\mathcal{F}^+,\mathcal{F}^-)$ with $d_i>2$. By the Poincare-Hopf Theorem, $\|\alpha\|= \langle \alpha,\sum_{i=1}^t \frac{d_i-2}{2} u \rangle$ holds for all integer classes $\alpha\in C$, so it holds for any $\alpha\in C$. Since $\|[S]\|=|\chi(S)|=(|\chi(S)|u)\cap [S]$, the dual of fibered face $F$ is $\sum_{i=1}^t\frac{d_i-2}{2} u=|\chi(S)|u$.
\end{proof}

Decompose $H_1(N;\mathbb{Z})/Tor=\mathbb{Z}[u]\oplus T$ as usual, here $T$ is the image of $H_1(S;\mathbb{Z})$. Let $\hat{p}:\hat{N}\rightarrow N$ be the maximal free abelian cover of $N$, $\hat{S}$ be one component of $\hat{p}^{-1}(S)$, and $\hat{\tau}\subset \hat{S}$ be one component of $\hat{p}^{-1}(\tau)$. Then we have a $T$ action on $\hat{\tau}$, so branches and switches of $\hat{\tau}$ give us $\mathbb{Z}[T]$-modules $\mathbb{Z}[T]^E$ and $\mathbb{Z}[T]^V$. Since $\hat{\phi}(\hat{\tau})$ is carried by $\hat{\tau}$, we have an $\hat{\phi}$ action on these two modules with matrices $P_E(t)$ and $P_V(t)$ respectively.

\begin{defn}
We define polynomial $\Phi$ by: $$\Phi=\frac{det(uI-P_E(t))}{det(uI-P_V(t))}=\sum_g a_g\cdot g.$$
\end{defn}

\begin{rem}
Here $\tau$ is only a bigon track but may not be a train track, so $\Phi$ may not be equal to $\Theta_F$.
\end{rem}

\begin{prop}\label{simplify}
For any $\alpha \in C$, we have $\lambda(\alpha)=\sup\{k>1|\  0=\Phi_F(k^{\alpha})=\sum_g a_g\cdot k^{\langle \alpha, g\rangle}\}$.
\end{prop}

\begin{proof}
Let $N'$ be as in the proof of Lemma \ref{nice1}, let $H_1(N';\mathbb{Z})/Tor=\mathbb{Z}[u]\oplus T'$ with $T'$ correspond with image of $H_1(S';\mathbb{Z})$. Then we have exact sequence $0\rightarrow K\rightarrow H_1(N';\mathbb{Z})/Tor \rightarrow H_1(N;\mathbb{Z})/Tor \rightarrow 0$, here $K\subset T'$.

Let $\Theta_{F'}$ be the Teichmuller polynomial associated with fibered cone $C'\subset H^1(N';\mathbb{R})$ containing $[S']$. Since $\tau'\subset S'$ is an invariant train track of $\phi'$, we can use $\tau'$ to compute $\Theta_{F'}$. Let $\tilde{p}:\tilde{N'}\rightarrow N'$ be the maximal free abelian cover of $N'$, $\tilde{S}'$ be one component of $\tilde{p}^{-1}(S')$ and $\tilde{\tau'}=\tilde{p}^{-1}(\tau) \cap \tilde{S}'$. Then by Theorem \ref{formula}, we have $$\Theta_{F'}=\frac{det(uI-P_{E'}(t'))}{det(uI-P_{V'}(t'))}.$$ Here $P_{E'}(t')$ and $P_{V'}(t')$ are the induced map of $\tilde{\phi}':\tilde{S}'\rightarrow \tilde{S}'$ on the $\mathbb{Z}[T']$-modules of branches and switches of $\tilde{\tau'}$.

Since $\lambda(i^*(\alpha))=\lambda(\alpha)$ by Lemma \ref{simplify}, to compute $\lambda(\alpha)$ for $\alpha \in C$, we do not need to use the full power of $\Theta_{F'}$, but only need to take valuate of $\Theta_{F'}$ on $i^*(H^1(N;\mathbb{R}))$.
Let $\Theta_{F'}=\sum_g a_g\cdot g$, here $g\in H_1(N';\mathbb{Z})/Tor$, then for any $i^*(\alpha) \in C'$, by Theorem \ref{dilatation} we have that $\lambda(i^*(\alpha))$ is the greatest root of $$0=\Theta_{F'}(X^{i^*(\alpha)})=\sum_g a_g\cdot X^{\langle i^*(\alpha),g\rangle}=\sum_g a_g\cdot X^{\langle \alpha,i_*(g) \rangle}.$$
So for computing dilatation function $\lambda(\cdot)$ on $C$, we need only to compute $P_{E'}(t')$ and $P_{V'}(t')$ modulo $K$.

Let $\hat{p}:\hat{N'}\rightarrow N'$ be the free abelian cover given by $\pi_1(N')\rightarrow H_1(N';\mathbb{Z})/Tor\rightarrow H_1(N;\mathbb{Z})/Tor$. Then there exist inclusion $\hat{i}:\hat{N'}\rightarrow \hat{N}$ to the maximal free abelian cover of $N$. $\hat{N'}$ is an intermediate cover of $\tilde{p}:\tilde{N'}\rightarrow N'$, while the deck transformation group of $p:\tilde{N'}\rightarrow \hat{N'}$ is $K$. Let $\hat{S'}=p(\tilde{S'})$, $\hat{\tau'}=p(\tilde{\tau'})$ and $\hat{\phi}':\hat{S'}\rightarrow \hat{S'}$ be a lift of $\phi'$. Then we have $P_{E'}(t')=\hat{P}_E'(t')$ and $P_{V'}(t')=\hat{P}_V'(t')$ modulo $K$, here $\hat{P}_E'(t')$ and $\hat{P}_V'(t')$ are the matrices of $\hat{\phi}'$ action on the $\mathbb{Z}[T]$-module of branches and switches of $\hat{\tau'}$.

Since $\hat{\tau'}$ is isomorphic with $\hat{\tau}$ with $T$ and $\phi$ action, we have $\hat{P}_E'(t')=P_E(t)$ and $\hat{P}_V'(t')=P_V(t)$. So we have $\Theta_{F'}(X^{i^*(\alpha)})=\Phi(X^{\alpha})$, and the Proposition is true by Theorem \ref{dilatation}.
\end{proof}

Now we try to compute $\Phi$ from the two families of curves $\{a_i\}_{i=1}^m,\{b_j\}_{j=1}^n$ and presentation of $\phi$ directly.

Let $q:\hat{S}\rightarrow S$ be the covering map induced by the maximal abelian cover of $3$-manifold. Since $\phi$ is generic, each component of $q^{-1}(a_i)$ ($q^{-1}(b_j)$) is mapped to $a_i$ ($b_j$) by homeomorphism. After choosing a component of $q^{-1}(a_i)$ ($q^{-1}(b_j)$) to be $\hat{a}_i$ ($\hat{b}_j$), we have $q^{-1}(a_i)=T\cdot \hat{a}_i$ ($q^{-1}(b_j)=T\cdot \hat{b}_j$). Then we can construct a weighted graph $G$ with $T$ action, called the {\it intersection graph}. Each vertex of the intersection graph $G$ corresponds with a curve in $(\cup_{i=1}^m T\cdot \hat{a}_i)\cup(\cup_{j=1}^n T\cdot \hat{b}_j)$. Two vertices $x,y$ in $G$ are connected by an edge $e=e(x,y)$ if the corresponding curves in $\hat{S}$ intersects, the weight $w(e(x,y))$ is the intersection number $\#(x\cap y)$.

Given these information, we construct a $m\times n$ matrix $M(t)$ named {\it intersection matrix}. The entries in $M(t)$ are elements in $\mathbb{Z}[T]$: $M_{i,j}=\sum_{t\in T} e(\hat{a}_i,t\hat{b}_j)\cdot t$. For any vector $v=(v_1,\cdots,v_m)$ with $v_i\in\mathbb{Z}_{\geq 0}$, let
$$M^v=\left( \begin{array}{cc}
I_{m\times m}& diag(v_1,\cdots,v_m)\cdot M(t)\\
0 & I_{n\times n}
\end{array}\right).$$
For $w=(w_1,\cdots,w_n)$ with $w_i\in\mathbb{Z}_{\geq 0}$, let
$$M_w=\left( \begin{array}{cc}
I_{m\times m}& 0\\
diag(w_1,\cdots,w_n)\cdot M^T(t^{-1}) & I_{n\times n}
\end{array}\right).$$

For vectors $v$ and $w$ as above, let Dehn-twists $T_{a_1}^{v_1}\cdots T_{a_m}^{v_m}$ and $T_{b_1}^{-w_1}\cdots T_{b_n}^{-w_n}$ denoted by $T_a^v$ and $T_b^{-w}$ respectively. Then any $\phi \in \mathcal{D}(a^+,b^-)$ can be conjugated to $T_a^{v_1}T_b^{-w_1}\cdots T_a^{v_s}T_b^{-w_s}$ with  $v_i\ne \vec{0},\ w_i\ne \vec{0}$ for any $i$. With these notations, we have the following formula for $\Phi$.

\begin{prop}\label{Phi}
Suppose $\phi=T_a^{v_1}T_b^{-w_1}\cdots T_a^{v_s}T_b^{-w_s}\in \mathcal{D}(a^+,b^-)$ is a generic pseudo-Anosov map, and $\{a_i\}_{i=1}^m$ intersects $\{b_j\}_{j=1}^n$ at $r$ points. Then $\Phi=(u-1)^{r-m-n}\cdot det(uI-M_{w_s}M^{v_s}\cdots M_{w_1}M^{v_1}).$
\end{prop}

\begin{proof}
Let $\hat{T}_{a_i}$, $\hat{T}^{-1}_{b_j}$ be the lift of $T_{a_i}$, $T^{-1}_{b_j}$ respectively. For each $\hat{a}_i\subset\hat{\tau}$ ($\hat{b}_j\subset \hat{\tau}$), we choose a branch $x_{i,1}\subset\hat{a}_i$ ($y_{j,1}\subset\hat{b}_j$). Let $x_{i,2},\cdots,x_{i,k_i}$ be the other branches of $\hat{\tau}$ lying on $\hat{a}_i$, similarly for $\hat{b}_j$.

Then for $[x_{i,1}]-[x_{i,k}]\in \mathbb{Z}[T]^E$, we have $(\hat{T}^{-1}_{b_j})^*([x_{i,1}]-[x_{i,k}])=[x_{i,1}]-[x_{i,k}]$ and $(\hat{T}_{a_{i'}})^*([x_{i,1}]-[x_{i,k}])=[x_{i,1}]-[x{i,k}]$ if $i\ne i'$. Moreover, $(\hat{T}_{a_{i}})^*[x_{i,1}]=[x_{i,1}]+\sum_b [b]$, here $b$ runs over all the branches of $\hat{\tau}$ that intersects $\hat{a}_i$ to the left (or right, which depends on the presentative of $\hat{T}_{a_{i}}$). At the same time, we also have $(\hat{T}_{a_{i}})^*[x_{i,k}]=[x_{i,k}]+\sum_b [b]$, so $(\hat{T}_{a_{i}})^*([x_{i,1}]-[x_{i,k}])=[x_{i,1}]-[x_{i,k}]$. Now we have $\hat{\phi}^*([x_{i,1}]-[x_{i,k}])=[x_{i,1}]-[x_{i,k}]$. The same argument hold for $[y_{j,1}]-[y_{j,l}]$.

Let $V_1$ be the $\mathbb{Z}[T]$-submodule of $\mathbb{Z}[T]^E$ freely generated by $[x_{i,1}]-[x_{i,k}]$ and $[y_{j,1}]-[y_{j,l}]$, here $i\in\{1,\cdots,m\}, k\in\{2,\cdots,k_i\}, j\in\{1,\cdots,n\}, l\in\{2,\cdots,l_j\}$. $V_1$ is a $\hat{\phi}^*$-invariant submodule of $\mathbb{Z}[T]^E$ with $\hat{\phi}^*|_{V_1}=id_{V_1}$. The $\mathbb{Z}[T]$-dimension of $V_1$ is $ \# \{$branches of $\tau\}-m-n$. It is easy to check that $\tau$ has $2r$ branches, so $dim_{\mathbb{Z}[T]}V_1=2r-m-n$.

To simplify notions, let $x_i=x_{i,1}$ and $y_j=y_{j,1}$. We have another $\mathbb{Z}[T]$-submodule $V_2$ of $\mathbb{Z}[T]^E$ which is freely generated by $x_1,\cdots,x_m$ and $y_1,\cdots,y_n$. then $\mathbb{Z}[T]^E=V_1\oplus V_2$. We have $(\hat{T}^{-1}_{b_j})^*([x_i])=[x_i]$ and $(\hat{T}_{a_{i'}})^*([x_i])=[x_i]$ if $i\ne i'$. Moreover, $(\hat{T}_{a_i})^*([x_i])=[x_i]+\sum_{j=1}^m M_{i,j}(t)[y_j]\ mod\ V_1$. Similarly, $(\hat{T}_{b_j})^*([y_j])=[y_j]+\sum_{i=1}^n M_{j,i}(t^{-1})[x_i]\ mod\ V_1$, and all the other Dehn-twists act trivially on $y_j$. These actions coincide with the action of $M^v$ and $M_w$. Since $\phi=T_a^{v_1}T_b^{-w_1}\cdots T_a^{v_s}T_b^{-w_s}$ is composition of Dehn-twists, the action matrix of $\hat{\phi}^*$ on $\mathbb{Z}[T]^E/V_1$ is given by $M_{w_s}M^{v_s}\cdots M_{w_1}M^{v_1}$. So $det(uI-P_E(t))=(u-1)^{2r-m-n}\cdot det(uI-M_{w_s}M^{v_s}\cdots M_{w_1}M^{v_1})$.

Since $\hat{\phi}$ fixes each switch of $\hat{\tau}$, we have $\hat{\phi}^*=id_{\mathbb{Z}[T]^V}$ on switches. So $det(uI-P_V(t))=(u-1)^r$, and $\Phi=(u-1)^{r-m-n}\cdot det(uI-M_{w_s}M^{v_s}\cdots M_{w_1}M^{v_1})$.
\end{proof}

An interesting property is, if we twist the $a$ curves and $b$ curves proportionally each time, then the dilatation function is symmetric:

\begin{prop}\label{algsymmetry}
Let $v=(v_1,\cdots,v_m),\ w=(w_1,\cdots,w_n)$ be two vectors with $v_i,w_j\in \mathbb{Z}_+$, and $x_k,y_k\in \mathbb{Z}_+,i=1,\cdots,s$. Suppose $\phi=T_a^{x_1v}T_b^{-y_1w}\cdots T_a^{x_sv}T_b^{-y_sw}\in \mathcal{D}(a^+,b^-)$ is a generic pseudo-Anosov map on surface $S$. Then $\Phi(u,t)=\Phi(u,t^{-1})$ and $m_F=\frac{[S]}{|\chi(S)|}$.
\end{prop}
\begin{proof}
Let $D=diag(v_1,\cdots,v_m)$ and $D'=diag(w_1,\cdots,w_n)$

By Proposition \ref{Phi}, $\Phi(u,t)=(u-1)^{r-m-n}\cdot det(uI-(M_{y_sw}M^{x_sv})\cdots (M_{y_1w}M^{x_1v}))$. By direct computation, $$M_{y_iw}M^{x_iv}=\left( \begin{array}{cc}
I_{m\times m}& x_iD\cdot M(t)\\
y_i D'\cdot M^T(t^{-1}) & I_{n\times n}+x_iy_iD'\cdot M^T(t^{-1})\cdot D\cdot M(t)
\end{array}\right).$$

By induction, we can show that $$(M_{y_sw}M^{x_sv})\cdots (M_{y_1w}M^{x_1v})=
\left(
\begin{array}{cc}
I_{m\times m} +E(t)& F(t)\\
G(t) & I_{n\times n}+H(t)
\end{array}\right)$$
$$=\left(
\begin{array}{cc}
I_{m\times m} +P_1(D\cdot M(t)\cdot D' \cdot M^T(t^{-1}))& D\cdot M(t)\cdot P_2(D'\cdot M^T(t^{-1})\cdot D\cdot M(t))\\
D'\cdot M^T(t^{-1})\cdot P_3(D\cdot M(t)\cdot D' \cdot M^T(t^{-1})) & I_{n\times n}+P_4(D'\cdot M^T(t^{-1})\cdot D\cdot M(t))
\end{array}\right),$$
here $P_1,P_2,P_3,P_4$ are polynomials.

So $$\frac{\Phi(u,t)}{(u-1)^{r-m-n}}= det \left(
\begin{array}{cc}
(u-1)I_{m\times m}-E(t)& -F(t)\\
-G(t) & (u-1)I_{n\times n}-H(t)
\end{array}\right)$$
$$=det((u-1)I_{m\times m} -E(t))\cdot det((u-1)I_{n\times n}-H(t)-G(t)\cdot[(u-1)I_{m\times m}-E(t)]^{-1}\cdot F(t)).$$

By taking transpose of the matrix, we get
$$\frac{\Phi(u,t)}{(u-1)^{r-m-n}}= det \left(
\begin{array}{cc}
(u-1)I_{m\times m} -E^T(t)& -G^T(t)\\
-F^T(t) & (u-1)I_{n\times n}-H^T(t)
\end{array}\right)$$
$$=det((u-1)I_{m\times m}-E^T(t))\cdot det((u-1)I_{n\times n}-H^T(t)-F^T(t)\cdot[(u-1)I_{m\times m} -E^T(t)]^{-1}\cdot G^T(t)).$$

We have expression $$E(t)=P_1(D\cdot M(t)\cdot D' \cdot M^T(t^{-1})),\ F(t)=D\cdot M(t)\cdot P_2(D'\cdot M^T(t^{-1})\cdot D\cdot M(t)),$$
$$G(t)=D'\cdot M^T(t^{-1})\cdot P_3(D\cdot M(t)\cdot D' \cdot M^T(t^{-1})), \ H(t)=P_4(D'\cdot M^T(t^{-1})\cdot D\cdot M(t)).$$

Using the expression above, we can check that
$$det((u-1)I_{m\times m} -E(t))=det((u-1)I_{m\times m} -E^T(t^{-1})),$$ and
$$det((u-1)I_{n\times n}-H(t)-G(t)\cdot[(u-1)I_{m\times m} -E(t)]^{-1}\cdot F(t))=$$
$$det((u-1)I_{n\times n}-H^T(t^{-1})-F^T(t^{-1})\cdot[(u-1)I_{m\times m} -E^T(t^{-1})]^{-1}\cdot G^T(t^{-1})).$$ So $\Phi(u,t)=\Phi(u,t^{-1})$.

Let $(u,t_1,\cdots,t_{b-1})$ be a basis of $H_1(N;\mathbb{Z})/Tor$, with $([S],\alpha_1,\cdots,\alpha_{b-1})$ be the dual basis of $H^1(N;\mathbb{Z})$. By Lemma \ref{nice1}, we have $F\subset \{\frac{1}{|\chi(S)|}[S]+x_1\alpha_1+\cdots+x_{b-1}\alpha_{b-1}|x_i\in\mathbb{R}\}$. Since $\Phi(u,t)=\Phi(u,t^{-1})$, for any $x=x_1\alpha_1+\cdots+x_{b-1}\alpha_{b-1}$ and $\frac{1}{|\chi(S)|}[S]+x\in F$, we have $\lambda(\frac{1}{|\chi(S)|}[S]+x)=\lambda(\frac{1}{|\chi(S)|}[S]-x)$. By the uniqueness of minimal point $m_F$, $m_F=\frac{1}{|\chi(S)|}[S]$.

\end{proof}

\begin{rem}
In Proposition \ref{algsymmetry}, the minimal point $m_F$ is rational since the dilatation function $\lambda(\cdot)$ is symmetric on fibered face $F$. However, the symmetric property of $\lambda(\cdot)$ here does not have an immediate geometric interpretation as in Proposition \ref{symmetry}.
\end{rem}

\subsection{An Irrational Example from Penner's Construction}\label{numerical2}
Now we give an example which shows that if the Dehn-twists on $a$-curves and $b$-curves are not proportional as in Proposition \ref{algsymmetry}, then the minimal point $m_F$ may not be rational.

Let surface $S$ and the two families of curves $\{a_1,a_2,a_3\}$ and $\{b_1,b_2\}$ be as shown in Figure 7. Since the $b$-curves are quite complicated, we draw $b_1$ as a red curve and $b_2$ as a pink curve.

\begin{center}
\psfrag{a}[]{\color{blue} $a_1$} \psfrag{b}[]{\color{blue} $a_2$} \psfrag{c}[]{\color{blue} $a_3$} \psfrag{d}[]{\color{red} $b_1$} \psfrag{e}[]{\color{magenta} $b_2$} \psfrag{t}[]{$t$}
\includegraphics[width=5in]{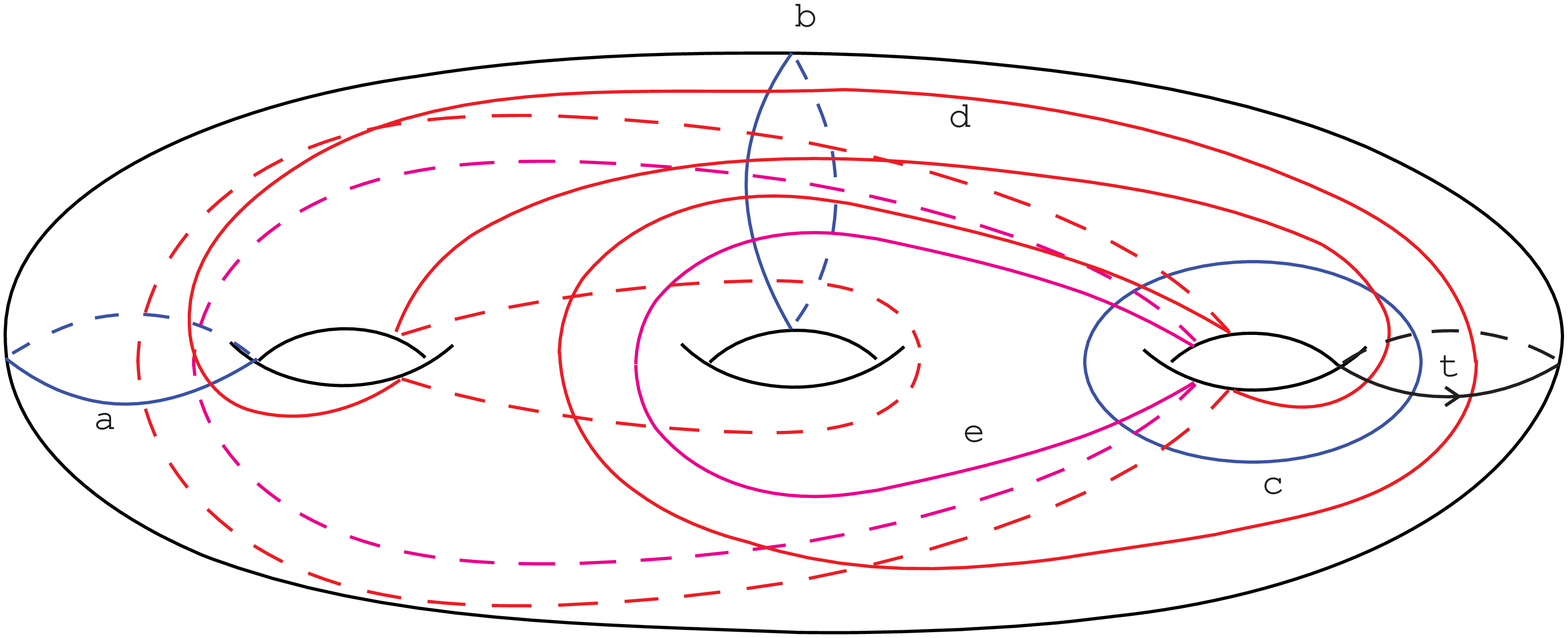}
\vskip 0.5 truecm
 \centerline{Figure 7}
\end{center}

Let's take pseudo-Anosov map $\phi=T_a^{(1,1,1)}\cdot T_b^{(-1,-1)}\cdot T_a^{(2,1,2)}\cdot T_b^{(-2,-1)}\in \mathcal{D}(a^+,b^-)$, for which the Dehn-twists along $a$-curves and $b$-curves are not proportional.

We can check that $\phi$ is generic. For the mapping torus $N=M(S,\phi)$, $H_1(N;\mathbb{Z})/Tor=\mathbb{Z}[u]\oplus \mathbb{Z}[t]$. Here $u$ is given by Lemma \ref{nice1} and $t$ is as shown in Figure 7.

Let $\tilde{N}$ be the maximal free abelian cover of $N$, and $\tilde{S}$ be one component of the preimage of $S$. For the preimage of $a$-curves and $b$-curves in $\tilde{S}$, we can get the intersection graph $G$, which is shown in Figure 8.

\begin{center}
\psfrag{a}[]{\color{blue} $a_1$} \psfrag{b}[]{\color{blue} $a_2$} \psfrag{c}[]{\color{blue} $a_3$} \psfrag{d}[]{\color{red} $b_1$} \psfrag{e}[]{\color{magenta} $b_2$} \psfrag{f}[]{\color{blue} $t^{-1}a_1$} \psfrag{g}[]{\color{blue} $t^{-1}a_2$} \psfrag{h}[]{\color{blue} $t^{-1}a_3$} \psfrag{i}[]{\color{magenta} $t^{-1}b_2$} \psfrag{j}[]{\color{blue} $ta_2$} \psfrag{k}[]{\color{red} $tb_1$}
\psfrag{l}[]{$1$} \psfrag{m}[]{$2$} \psfrag{n}[]{$4$} \psfrag{o}[]{$\cdots$}
\includegraphics[width=5.5in]{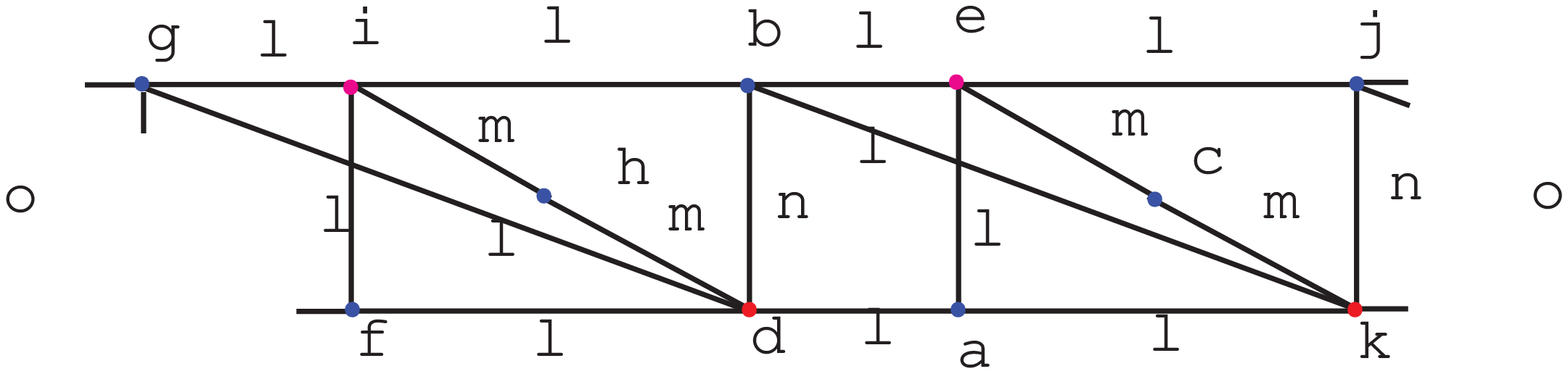}
\vskip 0.5 truecm
 \centerline{Figure 8}
\end{center}

From Figure 8, we can get the intersection matrix:
$M(t)=\left(
\begin{array}{cc}
t+1 & 1\\
t+4 & 1+t^{-1}\\
2t & 2
\end{array}\right).$

By Proposition \ref{Phi} and direct computation, we get
$\frac{\Phi(u,t)}{(u-1)^{14-5}}=$\\
$det(uI-M_{(2,1)}M^{(2,1,2)}M_{(1,1)}M^{(1,1,1)})=(u-1)\cdot (u^4-(78t^2+785t+1929+779t^{-1}+77t^{-2})u^3+(25t^2+2673t+21326+2673t^{-1}+25t^{-2})u^2-(77t^2+779t+1929+785t^{-1}+78t^{-2})u+1).$

For this example, $\Phi(u,t)\ne \Phi(u,t^{-1})$, so the dilatation function is not symmetric as in Proposition \ref{algsymmetry}. Let $(\alpha_1,\alpha_2)$ be a basis of $H^1(N;\mathbb{Z})$ dual with $(u,t)$. Then by Theorem \ref{cone}, Lemma \ref{nice1} and the formula of $\Phi(u,t)$ above, we get the fibered face $F=\{\frac{1}{4}\alpha_1+s\alpha_2|\ s\in (-1/8,1/8)\}$.

Let the minimal point $m_F=\frac{1}{4}\alpha_1+s\alpha_2$, by solving the equations numerically by Mathematica (\cite{Math}), we have $s=0.0001117568645\dots$, while $\frac{1}{4s}=2236.999051\dots$. However, by the algebraic number theory method as in section 5.2, if $\frac{1}{4s}$ is a rational number, its denominator must be less or equal to $40$, which is impossible. So $A(S,\phi)\ne \mathbb{Q}$ here.

Now let's take cohomology class $n\alpha_1+\alpha_2$ with $n\geq 3$, $n\alpha_1+\alpha_2$ lies in the fibered cone $C$. While $\|n\alpha_1+\alpha_2\|=4n$, which is represented by a closed surface with genus $2n+1$ with $n\geq 3$.

On the other hand, for the example $(S,\phi)$ we constructed in this subsection, $S$ has a double cover $S'$ given by $\theta: H_1(S;\mathbb{Z})\rightarrow \mathbb{Z}_2$. Here $\theta$ is defined by $\theta([t])= \bar{1}$ while $\theta(a_i)=\theta(b_j)=\bar{0}$. It is easy to check that $\phi$ lifts to $\phi':S'\rightarrow S'$. Since $S'$ has genus $5$ and $A(S,\phi)$ is covering invariant (Proposition \ref{cover}), we get a genus $5$ example with $A(S,\phi)\ne\mathbb{Q}$.

So we have the following theorem parallel with Theorem \ref{even}.
\begin{thm}\label{odd}
For any closed genus $2n+1$ surface $S$ with $n\geq 1$, there exists pseudo-Anosov map $\phi$ on $S$, such that $A(S,\phi)\ne\mathbb{Q}$.
\end{thm}

With Theorem \ref{onecusp}, \ref{even} and \ref{odd} together, we get Theorem \ref{general} in the introduction.

\section{Other Examples}\label{numerical3}
Proposition \ref{special} tells us on surfaces $\Sigma_{0,4}$, $\Sigma_{1,2}$ and $\Sigma_{2,0}$, $A(S,\phi)=\mathbb{Q}$ always holds for any pseudo-Anosov map $\phi$. However, Theorem \ref{general} tells us there exists pseudo-Anosov map $\phi$ on all the possible closed and one punctured surfaces with $A(S,\phi)\ne \mathbb{Q}$. The remaining "small" surfaces are $\Sigma_{0,5}$ and $\Sigma_{1,3}$, both of them have Euler characteristic $-3$.

In this section, we will give two pseudo-Anosov maps on $\Sigma_{0,5}$ and $\Sigma_{1,3}$, with $A(S,\phi)\ne \mathbb{Q}$. The method to show irrationality is same with the method in section 5.2. Since the computation is routine and tedious, we will not give the proof here.

The example for $\Sigma_{1,3}$ is given by two families of filling curves $\{a_1,a_2\}, \{b_1,b_2\}$ in Figure 9 (a), here $a_2$ and $b_1$ both bound a three punctured sphere. Let $t_{a_2}$ be the left hand half Dehn twist along $a_2$. $t_{a_2}$ is a homeomorphism of the three punctured sphere, being identity on a neighborhood of $a_2$ and do a $\pi$-rotation away from a bigger neighborhood of $a_2$ and exchanges the other two punctures. $t_{b_1}$ is defined similarly. The pseudo-Anosov map we take is $\phi=T_{a_1}t_{a_2}t^{-1}_{b_1}T^{-1}_{b_2}T^2_{a_1}t_{a_2}t^{-3}_{b_1}T^{-1}_{b_2}$. $\phi$ does not come from Penner's construction, since it contains half twist. However, Penner's construction of bigon track still works here. We define $\hat{\mathcal{D}}(a^+,b^-)$ to be augmented semigroup of $\mathcal{D}(a^+,b^-)$, which also contains half Dehn twists, then $\phi\in \hat{\mathcal{D}}(a^+,b^-)$. This $\phi$ is generic , and we can still use $\Phi$ to compute the dilatation function.

The example for $\Sigma_{0,5}$ is a pseudo-Anosov map $\phi$ given by braid as shown in Figure 9 (b) with word $\sigma_1\sigma_3\sigma_2^{-1}$. This $\phi$ is generated by half Dehn-twists along two families of curves $a^+$ and $b^-$, so it lies in $\hat{\mathcal{D}}(a^+,b^-)$. However, this $\phi$ is not generic. We can construct the invariant train track explicitly and compute $\Theta_F$. In this example, we need to compute Alexander Polynomial of the manifold to compute the Thurston norm (see \cite{McM1} Theorem 1.1).

\begin{center}
\psfrag{a}[]{\color{blue} $a_1$} \psfrag{b}[]{\color{blue} $a_2$} \psfrag{c}[]{\color{red} $b_1$} \psfrag{d}[]{\color{magenta} $b_2$} \psfrag{e}[]{(b)} \psfrag{f}[]{(a)}
\includegraphics[width=5.5in]{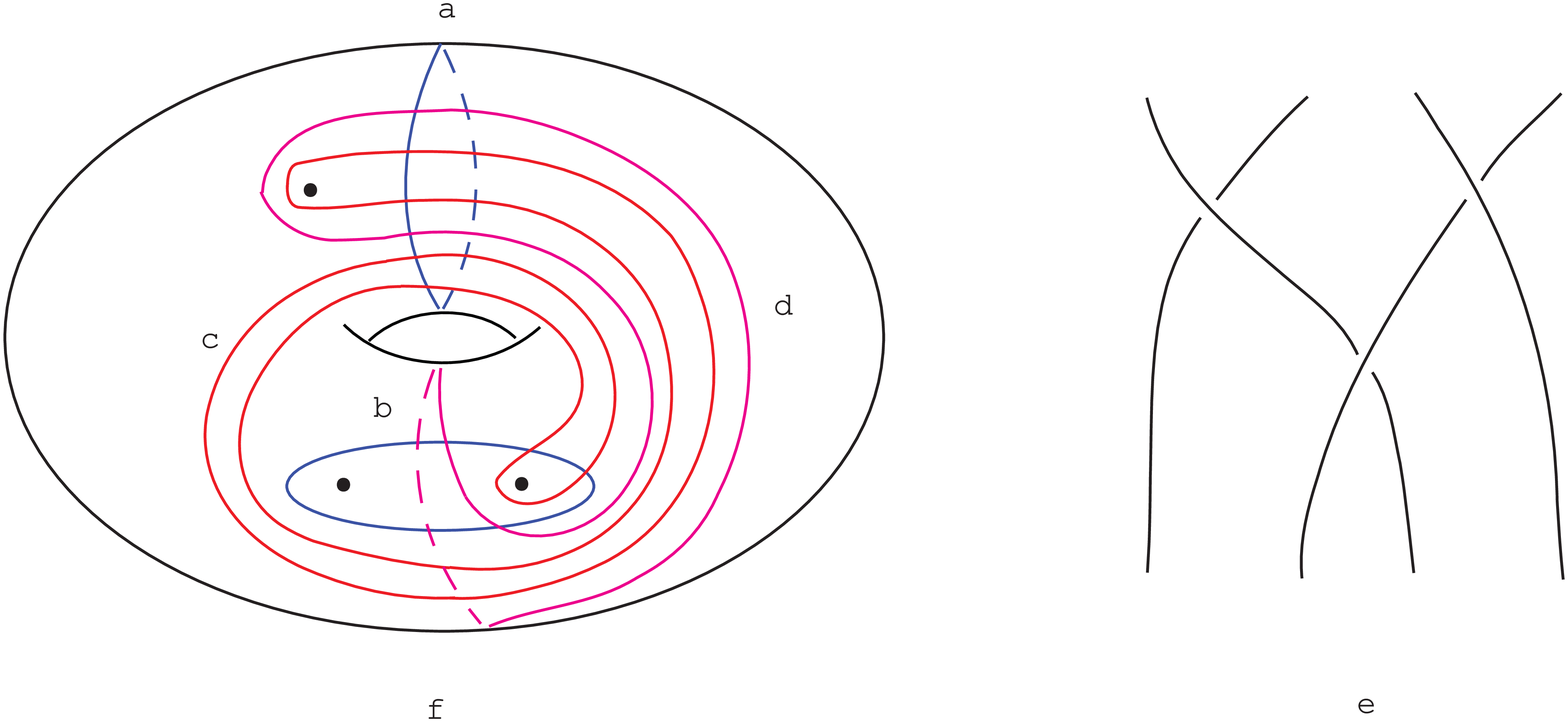}
\vskip 0.5 truecm
 \centerline{Figure 9}
\end{center}

\section{$A(S,\phi)$ is not Defined for Manifold}\label{globalsec}

We defined $A(S,\phi)$ as an invariant of pseudo-Anosov map, and also an invariant $A_C$ for fibered cone $C\subset H^1(N;\mathbb{R})$.
So an natural question is, for different fibered cones $C_1,C_2$ of the same manifold $N$, whether $A_{C_1}= A_{C_2}$ always holds. In this section, we will give an example to show the above equality does not necessarily hold.

Our example comes from Dehn-filling of one cusp of the magic manifold $N$. The magic manifold is the complement of a $3$-component link in Figure 10 (a). In \cite{KKT}, the authors are interested in dilatation of monodromy of those surface bundles obtained by Dehn-filling of the magic manifold. From those examples, they got new examples of small dilatation pseudo-Anosov maps and estimated the asymptotic property of minimal dilatation pseudo-Anosov maps.

Let $F_{\alpha},\ F_{\beta},\ F_{\gamma}$ be the two punctured discs bounded by $\alpha,\ \beta,\ \gamma$ respectively, they are the obvious two punctured discs lying on the paper in Figure 10(a) and their orientation point out of the paper. Then $([F_{\alpha}],[F_{\beta}],[F_{\gamma}])$ is a basis of $H_2(N,\partial N;\mathbb{Z})=\mathbb{Z}^3$. An element $x[F_{\alpha}]+y[F_{\beta}]+x[F_{\gamma}]\in H_2(N,\partial N;\mathbb{Z})\cong H^1(N;\mathbb{Z})$ is written as $(x,y,z)$.

In section 2 of \cite{KKT}, the topological property of Dehn-filling of $\beta$ curve is studied. Let's take the $-\frac{7}{2}$ Dehn-filling of $\beta$ curve, the resulted manifold is denoted by $N(-\frac{7}{2})$. Here the $-\frac{7}{2}$ slope is given by the intersection of $(7,2,0)\in H_2(N,\partial N;\mathbb{Z})$ with the boundary torus of $N(\beta)$.
Then by Lemma 2.15 of \cite{KKT}, $a=(4,2,3),\ b=(3,2,4)\in H_2(N,\partial N;\mathbb{Z})$ give a basis of $H_2(N(-\frac{7}{2}),\partial N(-\frac{7}{2});\mathbb{Z})$, which is also denoted by $a$ and $b$.

By Lemma 2.16 of \cite{KKT}, the ball of $H_2(N(-\frac{7}{2}),\partial N(-\frac{7}{2});\mathbb{Z})$ with Thurston norm $4$ is a hexagon as in Figure 10 (b), with vertices $\pm(2,-1),\ \pm(1,-2),\ \pm(2,-2)$. All the faces of the hexagon in Figure 10 (b) are fibered faces.

\begin{center}
\psfrag{a}[]{$\alpha$} \psfrag{b}[]{$\beta$} \psfrag{c}[]{$\gamma$} \psfrag{d}[]{(a)} \psfrag{e}[]{(b)} \psfrag{f}[]{\color{red} $F_A$} \psfrag{s}[]{$F_S$} \psfrag{g}[]{$a$} \psfrag{h}[]{$b$}
\includegraphics[width=5.5in]{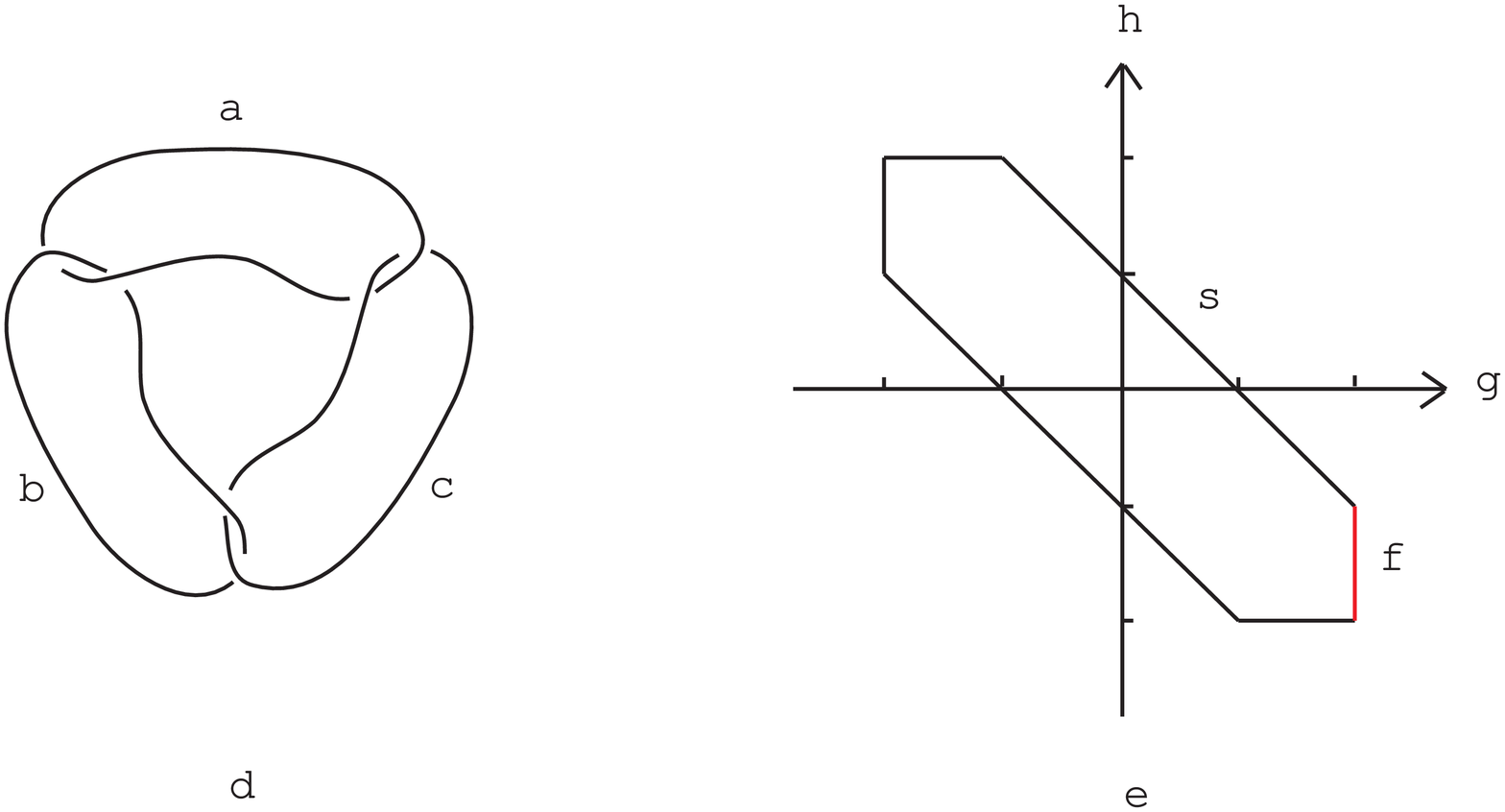}
\vskip 0.5 truecm
 \centerline{Figure 10}
\end{center}

The face $F_S$ is an {\it s-face} (symmetric face) as the definition in section 2.5.2 of \cite{KKT}. By Remark 3.3 of \cite{KKT}, the minimal point of $\lambda(\cdot)$ on $F_S$ is the middle point of $F_S$, which is rational. Let $C_S$ be the fibered cone associate with $F_S$, then $A_{C_S}=\mathbb{Q}$.

Now let's turn to think about the face $F_A$ as shown in Figure 10 (b), the two boundary points of $F_A$ is $(2,-1)$ and $(2,-2)$. These two points correspond to $(5,2,2)$ and $(2,0,-2)$ in $H^1(N;\mathbb{Z})$. By \cite{KT}, the fibered cone $C$ of magic manifold $N$ containing $(5,2,2)$ and $(2,0,-2)$ on the boundary has Teichmuller polynomial $\Theta=xyz^{-1}-x-y-xz^{-1}-yz^{-1}+1$, here $(x,y,z)$ is a basis of $H_1(N;\mathbb{Z})$ dual with $([F_{\alpha}],[F_{\beta}],[F_{\gamma}])$.

Let the minimal point be $m_F=t(2,-1)+(1-t)(2,-2)\in H^1(N(-\frac{7}{2});\mathbb{Z})$, which is equal to $(2+3t,2t,4t-2)\in H^1(N;\mathbb{Z})$. So $\lambda(m_F)$ is the greatest root of $0=X^{t+4}-X^{3t+2}-X^{2t}-X^{-t+4}-X^{-2t+2}+1$. By numerical computation of the minimal point by Mathematica (\cite{Math}), we get $t=0.528944\dots$, while $2/t=3.781116\dots$. By the algebraic number theory method, we get that, if $2/t$ is a rational number, the denominator is smaller or equal to $16$. But the numerical data implies it is impossible, so $A_{C_A}\ne \mathbb{Q}$.

In Section 2.5.2 of \cite{KKT}, they pointed out that all the A-faces of $N(r)$ are {\it entropy equivalent}. So for $N(-\frac{7}{2})$, one pair of fibered faces have $A_C=\mathbb{Q}$, while two other pairs have $A_C\ne \mathbb{Q}$ and these two pairs share the same invariant. So our invariant $A_C$ is not well defined for the whole manifold, but only for each fibered cone.

\begin{rem}
In Lemma 3.5 of \cite{KKT}, they have shown that all the fibered faces of $N(-\frac{5}{2})$ have rational minimal point. So we take $N(-\frac{7}{2})$ here, which is the "closest" example to $N(-\frac{5}{2})$.

In the proof of Lemma 3.6 of \cite{KKT}, they have tried to find the minimal point of A-face of $N(-\frac{2}{3})$ with computing $\lambda(\cdot)$ for some "very large" homology classes. It seems that they have realized the minimal point there may not be rational, but did not find a proper algebraic tool to prove it.
\end{rem}

\section{Further Questions}

1. Although $A(S,\phi)\ne \mathbb{Q}$ seems to be a general phenomenon, it is still quite mysterious what does this inequality mean. All the proof in this paper either does not give explicit example, or requires numerical computation to deduce irrationality. These two process both do not help the author much to understand the meaning of irrationality. A more sophisticated method to determine whether $A(S,\phi)=\mathbb{Q}$ will help us to understand the question much deeper. For all the examples we give in this paper, the rationality of minimal point comes from symmetry, i.e. the rational minimal point $m_F$ is the unique fixed point of an nice group action on the fibered face $F$. So another question is: whether rational $m_F$ implies any kind of symmetry?

2. Theorem \ref{general} tells us, for all closed or one punctured surfaces with $|\chi(S)|\geq 3$, there exist pseudo-Anosov $\phi$ on $S$ with $A(S,\phi)\ne\mathbb{Q}$. By the evident of drilling Theorem (Theorem \ref{drill}) and our numerical experiments, it is natural to guess that the same result hold for all surfaces with $|\chi(S)|\geq 3$. However, even if one single example of $A(S,\phi)\ne\mathbb{Q}$ gives us a family of such examples, we still do not have enough examples to show that $A(S,\phi)$ can be irrational for all surfaces with $|\chi(S)|\geq 3$.

3. The Drilling Theorem (Branched Covering Theorem) tells us drilling (branched covering) deduce irrationality for all but finitely many drilling (branched covering) classes. However, from the proof of the theorem, the number of "exceptional" drilling class is very large and depend on the Teichmuller polynomial numerically. The author would like to know, whether there is any better bound which only depend on the surface $S$ and dilatation $\lambda$, or even whether there is a universal bound.

4. In all the examples in this paper, the $\mathbb{Q}$-rank of $A(S,\phi)$ is either one or two. The author does not have an example with
higher rank $A(S,\phi)$ yet. Furthermore, whether $A(S,\phi)$ can be arbitrarily large? The author tends to believe that one can do drilling
construction to surface bundle with large $b_1$ to get large rank $A(S,\phi)$. The rank
of $A(S,\phi)$ also provides a simple invariant up to fibered cone commensurability, e.g. if $A(S,\phi)$ has large rank, then $(S,\phi)$ is not fibered cone commensurable with $(S',\phi')$ if Euler characteristic of $S'$ is small.

5. Since Virtually Fibered Conjecture holds (see \cite{Wi}, \cite{Ag}), every hyperbolic $3$-manifold with finite volume has a finite cover which is a surface bundle over circle. So for any commensurable class $\mathcal{C}$, the invariant $A_C$ is defined for a fibered cone $C$ of some $M\in\mathcal{C}$ . Then for any commensurability class $\mathcal{C}$ of hyperbolic $3$-manifold, we can ask many questions. For example, for any commensurability class $\mathcal{C}$, whether there is a surface bundle over circle $M\in \mathcal{C}$, which has a fibered face $C$, with $A_C=\mathbb{Q}$? For every commensurability class $\mathcal{C}$, we can also define a $\mathbb{Q}$-submodule of $\mathbb{R}$, which is defined by $A_{\mathcal{C}}=\sum A_C$, here $C$ runs over all the fibered cones of all the manifolds in the commensurability class $\mathcal{C}$. Since hyperbolic $3$-manifolds have virtually RFRS (residually finite rational solvable) fundamental group (see \cite{Ag1} and \cite{Ag}), $\mathcal{C}$ has virtually infinite first betti number. Moreover, for a fixed manifold $N$, we can always find a finite cover, which has a fibered cone does not correspond to any fibered cone of $N$. So $A_{\mathcal{C}}=\sum A_C$ is a sum of infinitely terms. Then it is natural to ask, whether $A_{\mathcal{C}}\ne \mathbb{Q}$ for any commensurability class $\mathcal{C}$? Moreover, we can also ask: whether $A_{\mathcal{C}}$ is always an infinitely generated $\mathbb{Q}$-submodule of $\mathbb{R}$?

\bibliographystyle{amsalpha}

\begin{thebibliography}{Irrational}

\bibitem[Ag1]{Ag1} {I. Agol}, {\it Criteria for virtual fibering}, J. Topol. 1 (2008), no. 2, 269 - 284.

\bibitem[Ag2]{Ag} {I. Agol}, {\it The virtual Haken conjecture}, arXiv:math.GT/1204.2810. With an appendix by I.Agol, D.Groves, J. Manning.

\bibitem[BH]{BH} {M. Bestvina, M. Handel}, {\it Train-tracks for surface homeomorphisms}, Topology 34 (1995), no. 1, 109 - 140.

\bibitem[CB]{CB} {A. Casson, S. Bleiler}, {\it Automorphisms of surfaces after Nielsen and Thurston}, vol. 9, London Math. Soc. Stud. Texts, Cambridge Univ. Press, Cambridge, 1988.

\bibitem[CSW]{CSW} {D. Calegari, H. Sun, S. Wang}, {\it On fibered commensurability}, Pacific J. Math. 250 (2011), no. 2, 287 - 317.

\bibitem[Fr]{Fr} {D. Fried}, {\it Flow equivalence, hyperbolic systems and a new zeta function for flows}, Comment. Math. Helv. 57 (1982), no. 2, 237 - 259.

\bibitem[FLP]{FLP} {\it Travaux de Thurston sur les surfaces}, S\'{e}minaire Orsay. With an English summary. Ast\'{e}risque, 66-67. Soci\'{e}t\'{e} Math\'{e}matique de France, Paris, 1979.

\bibitem[FM]{FM} {B. Farb, D. Margalit}, {\it A primer on mapping class groups}, Princeton Mathematical Series, 49. Princeton University Press,
Princeton, NJ, 2012.

\bibitem[FN]{FN} {N.I. Feldman, Yu.V. Nesterenko}, {\it Transcendental numbers. Number theory, IV}, Encyclopaedia Math. Sci., 44, Springer,
Berlin, 1998.

\bibitem[Ga]{Ga} {D. Gabai}, {\it Foliations and the topology of 3-manifolds}, J. Differential Geom. 18 (1983), no. 3, 445 - 503. 

\bibitem[KKT]{KKT} {E. Kin, S. Kojima, M. Takasawa}, {\it Minimal dilatations of pseudo-Anosovs generated by the magic 3-manifold and
their asymptotic behavior}, arXiv:math.GT/1104.3939.

\bibitem[KT]{KT} {E. Kin, M. Takasawa}, {\it Pseudo-Anosov braids with small entropy and the magic 3-manifold},  Comm. Anal. Geom. 19 (2011), no. 4, 705 - 758.

\bibitem[LO]{LO} {D. Long, U. Oertel}, {\it Hyperbolic surface bundles over the circle}, Progress in knot theory and related topics, 121 - 142,
Travaux en Cours, 56, Hermann, Paris, 1997.

\bibitem[Ma]{Ma} {S. Matsumoto}, {\it Topological entropy and Thurston's norm of atoroidal surface bundles over the circle}, J. Fac. Sci. Univ. Tokyo
Sect. IA Math. 34 (1987), no. 3, 763 - 778.

\bibitem[Math]{Math} {Wolfram Research, Inc.},  {\it Mathematica Edition: Version 8}, Wolfram Research, Inc., Champaign, IL (2008).


\bibitem[McM1]{McM} {C. McMullen}, {\it Polynomial invariants for fibered $3$-manifolds and Teichmuller geodesics for foliations}, (English, French summary) Ann. Sci. \'{E}cole Norm. Sup., (4) 33 (2000), no. 4, 519 - 560.

\bibitem[McM2]{McM1} {C. McMullen}, {\it The Alexander polynomial of a $3$-manifold and the Thurston norm on cohomology}, (English, French summary)
Ann. Scient. \'{E}cole Norm. Sup., (4) 35 (2002), no. 2, 153 - 171.

\bibitem[MR]{MR} {C. Maclachlan, A. Reid}, {\it The arithmetic of hyperbolic 3-manifolds}, Graduate Texts in Mathematics, 219. Springer-Verlag,
New York, 2003.

\bibitem[Ot]{Ot} {J.-P. Otal}, {\it The hyperbolization theorem for fibered $3$-manifolds}, SMF/AMS Texts and Monog. 7, Amer. Math. Soc.
(2001)   MR1855976   Translated from the 1996 French original by LD Kay.

\bibitem[Pe]{Pe} {R. Penner}, {\it A construction of pseudo-Anosov homeomorphisms}, Trans. Amer. Math. Soc. 310 (1988), no. 1, 179 - 197.


\bibitem[Th]{Th} {W.P. Thurston}, {\it A norm for the homology of $3$-manifolds}, Mem. Amer. Math. Soc. 339 (1986) 99 - 130.

\bibitem[Wi]{Wi} {D. Wise}, {\it The structure of groups with a quasiconvex hierarchy}, preprint.

\end{thebibliography}

\end{document}